\documentclass[12pt]{amsart}
\usepackage{amsmath}
\usepackage{amssymb}
\usepackage{amsthm}
\usepackage{array}
\usepackage{xy}
\usepackage[pdftex]{graphicx}
\usepackage{hyperref}
\usepackage{color}
\usepackage{transparent}
\usepackage{latexsym}

\numberwithin{equation}{section}

\newtheorem{thm}{Theorem}[section]
\newtheorem{cor}[thm]{Corollary}
\newtheorem{lem}[thm]{Lemma}
\newtheorem{dfn}[thm]{Definition}
\newtheorem{prop}[thm]{Proposition}
\newtheorem{rem}[thm]{Remark}

\begin{document}
\title[Free boundary regularity in the optimal partial transport problem]{Free boundary regularity in the optimal partial transport problem}
\author[E. Indrei]{Emanuel Indrei}

\maketitle


\def\signei{\bigskip\begin{center} {\sc Emanuel Indrei\par\vspace{3mm}
Department of Mathematics\\
The University of Texas at Austin\\
1 University Station, C1200\\
Austin TX 78712, USA\\
email:} {\tt eindrei@math.utexas.edu}
\end{center}}

%

%

%


%



\begin{abstract}
In the optimal partial transport problem, one is asked to transport a fraction $0<m \leq \min\{||f||_{L^1}, ||g||_{L^1}\}$ of the mass of $f=f \chi_\Omega$ onto $g=g\chi_\Lambda$ while minimizing a transportation cost. If $f$ and $g$ are bounded away from zero and infinity on strictly convex domains $\Omega$ and $\Lambda$, respectively, and if the cost is quadratic, then away from $\partial(\Omega \cap \Lambda)$ the free boundaries of the active regions are shown to be $C_{loc}^{1,\alpha}$ hypersurfaces up to a possible singular set. This improves and generalizes a result of Caffarelli and McCann \cite{CM} and solves a problem discussed by Figalli \cite[Remark 4.15]{Fi}. Moreover, a method is developed to estimate the Hausdorff dimension of the singular set: assuming $\Omega$ and $\Lambda$ to be uniformly convex domains with $C^{1,1}$ boundaries, we prove that the singular set is $\mathcal{H}^{n-2}$ $\sigma$-finite in the general case and $\mathcal{H}^{n-2}$ finite if $\Omega$ and $\Lambda$ are separated by a hyperplane.       
\end{abstract}

\section{Introduction}
Given two non-negative functions $f, g \in L^1(\mathbb{R}^n)$ and a number $m \leq \min\{||f||_{L^1}, ||g||_{L^1}\}$, the optimal partial transport problem consists of finding an optimal transference plan between $f$ and $g$ with mass $m$. In this context, a transference plan refers to a non-negative, finite Borel measure $\gamma$ on $\mathbb{R}^n \times \mathbb{R}^n$ with mass $m$ (i.e. $\gamma(\mathbb{R}^n \times \mathbb{R}^n)=m$) whose first and second marginals are controlled by $f$ and $g$ respectively: for any Borel set $A \subset \mathbb{R}^n$,
$$\gamma(A\times \mathbb{R}^n) \leq \int_{A} f(x)dx, \hskip .2in \gamma(\mathbb{R}^n \times A) \leq \int_{A} g(x)dx.$$
Let $\Gamma_{\leq}^m(f,g)$ denote the set of transference plans. By an optimal transference plan, we mean a minimizer of 
\begin{equation} \label{mini}
\inf_{\gamma \in\Gamma_{\leq}^m(f,g)} \int_{\mathbb{R}^n \times \mathbb{R}^n} |x-y|^2 d\gamma(x,y).
\end{equation}       

Issues of existence, uniqueness, and regularity of optimal transference plans have recently been addressed by Caffarelli \& McCann \cite{CM} and Figalli \cite{Fi}, \cite{Fi2}. By standard methods in the calculus of variations, one readily obtains existence of minimizers. However, in general, minimizers of (\ref{mini}) are far from unique. To see this, let $f\wedge g:= \min\{f,g\}$ and suppose $\mathcal{L}^n(supp (f \wedge g))>0$ (with $\mathcal{L}^n(\cdot):=|\cdot|$ being the Lebesgue measure and $supp(f \wedge g)$ the support of $f \wedge g$). Pick $$0<m < \int_{\mathbb{R}^n} (f \wedge g)(x) dx,$$and let $h<f \wedge g$ be any function with $||h||_{L^1(\mathbb{R}^n)}=m$. Note that the transference plan $\gamma_h:=(Id \times Id)_{\#}h$ is optimal (since its cost is zero). However, to construct this family of examples, one needs $\mathcal{L}^n(supp\{f \wedge g\})>0$. Indeed, under a disjointness assumption on the supports, Caffarelli and McCann \cite[Theorem 4.3]{CM} prove the existence of two domains $U_m \subset \Omega$, $V_m \subset \Lambda$ and a unique convex function $\Psi$ such that the unique minimizer of (\ref{mini}) is $\gamma:=(Id \times \nabla \Psi)_{\#}f \chi_{U_m}$, where $\nabla \Psi$ is the optimal transport between $f \chi_{U_m}$ and $g \chi_{V_m}$ \big($\overline{U_m \cap \Omega}$ and $\overline{V_m\cap \Lambda}$ are usually referred to as the active regions \big). Furthermore, by invoking Caffareli's regularity theory for the Monge-Amp\`{e}re equation \cite{C1}, \cite{C2}, \cite{C3}, \cite{C4}, the authors show that if $f$ and $g$ are supported on strictly convex domains separated by a hyperplane, then higher regularity on the densities implies higher regularity on $\Psi$ in the interior of the active region $\overline{U_m\cap \Omega}$ \cite[Theorem 6.2]{CM}. Moreover, employing a geometric approach, Caffarelli and McCann prove $\Psi \in C_{loc}^{1,\alpha}(\overline{\Omega \cap U_m} \setminus E)$ \cite[Corollary 7.14]{CM}, where $E \subset \partial \Omega$ is a possible singular set, and since $\nabla \Psi$ gives the direction of the normal to the free boundary $\overline{\partial U_m \cap \Omega}$ \cite[Corollary 7.15]{CM}, they also obtain local $C^{1,\alpha}$ regularity of the free boundary (symmetric arguments imply a similar statement for $\overline{\partial V_m \cap \Omega}$ -- the free boundary associated to $\Lambda$).

Figalli \cite{Fi} studies the case in which the disjointness assumption on the supports of the densities is removed. He proves that minimizers to (\ref{mini}) are unique for $$||f \wedge g||_{L^1(\mathbb{R}^n)}\leq m \leq \min\{||f||_{L^1(\mathbb{R}^n)}, ||g||_{L^1(\mathbb{R}^n)}\} \large,$$ \cite[Proposition 2.2 and Theorem 2.10]{Fi}. In fact, uniqueness of the partial transport is obtained for a general class of cost functions $c(x,y)$, dealing also with the case in which $f$ and $g$ are densities on a Riemannian manifold and $c(x,y)=d(x,y)^2$, where $d(x,y)$ is the Riemannian distance.

As in the disjoint case, Figalli obtains local interior $C^{0,\alpha}$ regularity of the partial transport (i.e. $\Psi \in C_{loc}^{1,\alpha}(U_m\cap \Omega))$ under some weak assumptions on the densities \cite[Theorem 4.8]{Fi}. However, in sharp contradistinction to the disjoint case, he constructs an example with $C^\infty$ densities for which the partial transport is not $C^1$, thereby showing that the interior $C_{loc}^{0,\alpha}$ regularity is in this sense optimal \cite[Remark 4.9]{Fi}. Furthermore, by assuming the densities to be bounded away from zero and infinity on strictly convex domains, he goes on to say that $\Psi$ has a $C^1$ extension to $\mathbb{R}^n$, and utilizing that $\nabla \Psi$ gives the direction of the normal to $\overline{\partial U_m \cap \Omega}$ (as in the disjoint case), he also derives local $C^1$ regularity of the free boundary away from $\partial(\Omega \cap \Lambda)$ \cite[Theorems 4.10 \& 4.11]{Fi}. 

                  
However, the author suggests that it may be possible to adapt the method of Caffarelli and McCann to prove H\"{o}lder regularity of the partial transport up to the free boundary \cite[Remark 4.15]{Fi}. As a direct corollary, one would thereby improve the $C_{loc}^{1}$ regularity of the free boundaries away from the common region into $C_{loc}^{1,\alpha}$ regularity. The first aim of the present work is to prove this result, see Corollary \ref{cor2}. Our method of proof follows the line of reasoning in Caffarelli and McCann \cite[Section 7]{CM}, although new ideas are needed to get around the lack of a separating hyperplane. Indeed, as mentioned earlier, Figalli's counterexample to $C^1$ regularity of the transport map in the non-disjoint case shows that the assumption of a separating hyperplane plays a crucial role in the regularity theory of the partial transport. 
The key part of our proof is the adaptation of the uniform localization lemma \cite[Lemma 7.11]{CM} (cf. Lemma \ref{d}). This is achieved by classifying the extreme points of the set $Z_{min}$ which comes up in the course of proving this lemma. Indeed, in the disjoint case, Caffarelli and McCann prove that the extreme points are in $\overline \Lambda$; however, this is insufficient to close the argument in the general case. To get around this difficulty, we make use of a theorem established by Figalli \cite[Theorem 4.10]{Fi}. Our method has the added feature of allowing us to identify, in a very specific way, the geometry of the singular set which comes up in the work of Caffarelli and McCann and prove the general uniform localization lemma under assumptions which in the disjoint case turn out to be weaker than the ones found in their work \cite[Lemma 7.11]{CM} (cf. Remark \ref{rem111}).     

The second aim of this paper is to prove that away from $\partial(\Omega \cap \Lambda)$, the free boundary intersects the fixed boundary in a $C^{1,\alpha}$ way up to a ``small" singular set. In the disjoint case, Caffarelli and McCann discovered that this set consists of nontransverse intersection points of fixed with free boundary and points that map to non-locally convex parts of the path-connected target region. Therefore, even in this case, one may not directly apply the implicit function theorem to obtain an estimate on its Hausdorff dimension. However, we exploit the geometry in the uniform localization lemma to prove that in addition to the above description, nontransverse singular points also have the property that when one shoots rays to infinity emanating from these points and in the direction of the normal to the boundary, the half-lines that are generated intersect the closure of the target region only along its boundary (see e.g. Figure \ref{xenon} and the set $X_s$ in Lemma \ref{d}). It turns out that this geometry is sufficient to connect the singular set with projections of convex sets onto other convex sets and prove a corresponding rectifiability result; this is the content of Proposition \ref{geom2b}.     

Mathematically, the previous discussion takes the following form: if the supports of the densities are separated by a hyperplane, then as previously mentioned, Caffarelli and McCann prove $\Psi \in C_{loc}^{1,\alpha}(\overline{\Omega \cap U_m} \setminus E)$, where $E \subset \partial \Omega$ is a closed set \cite[Corollary 7.15]{CM}. We generalize an improvement of this result to the non-disjoint case. Indeed, our result states that there exists a closed set $\tilde E \subset \partial \Omega \cup \partial(\Omega \cap \Lambda)$ for which $\Psi \in C_{loc}^{1,\alpha}(\overline{\Omega \cap U_m} \setminus \tilde E)$, and if $\overline{\Omega} \cap \overline{\Lambda}=\emptyset$, then $\tilde E \subset E$ (see Corollary \ref{cor1} and Remark \ref{frem}). Moreover, thanks to the general uniform localization lemma (Lemma \ref{d}), we are able to identify the set $\tilde E$ explicitly in terms of the geometry of $\Omega$ and $\Lambda$; using this information we prove that if the supports are uniformly convex with $C^{1,1}$ boundaries, then the singular set for the free boundaries is relatively closed (away from the common region $\Omega \cap \Lambda$) and $\mathcal{H}^{n-2}$ $\sigma$-finite in the general case and compact with $\mathcal{H}^{n-2}$ finite measure in the disjoint case; this is the content of Theorem \ref{thhm}.  
 
The paper is organized as follows: in \S \ref{Section2}, we fix some notation and introduce relevant ideas from the literature which will be useful in our analysis. \S \ref{Section3} is devoted to the $C_{loc}^{0,\alpha}$ regularity theory of the partial transport up to the free boundary; indeed, in this section we utilize the method of Caffarelli and McCann \cite[Section 7]{CM} to solve the problem mentioned by Figalli \cite[Remark 4.15]{Fi}. \S \ref{Section4} deals with the Hausdorff dimension of the singular set, and \S \ref{Section5} discusses several open problems.

\section{Preliminaries}
\label{Section2}
In this section, we will fix the notation for the remainder of the paper and state some of the relevant theorems from the literature.

\subsection{Notation} 

\begin{dfn} \label{projj} Given $\Omega \subset \mathbb{R}^n$ and a convex set $\Lambda \subset \mathbb{R}^n$, we denote the orthogonal projection of $\Omega$ onto $\Lambda$ by $P_\Lambda(\Omega)$.   
\end{dfn}

\noindent Note that in the special case when $\Omega \cap \Lambda = \emptyset$, $P_\Lambda(\Omega) \subset \partial \Lambda$. Hence, we understand $\partial P_\Lambda(\Omega)$ to be the boundary of $P_\Lambda(\Omega)$ seen as a subset of $\partial \Lambda$. In other words, $P_\Lambda(\Omega)$ is a manifold with boundary, and we denote the boundary by $\partial P_\Lambda(\Omega)$. In the general case, $\partial(P_\Lambda(\Omega) \cap \partial \Lambda)$ is defined in a similar way.       

\begin{dfn} \label{tang}
Given a $C^1$ set $\Lambda$, we denote the tangent space of $\Lambda$ at a point $y \in \partial \Lambda$ by $\mathbb{T}_y \Lambda$. Similar notation will be used if the set is Lipschitz.  
\end{dfn}

\begin{dfn} \label{cone}
Given an $(m-1)$-plane $\pi$ in $\mathbb{R}^m$, we denote a general cone with respect to $\pi$ by $$C_\alpha(\pi):= \{z \in \mathbb{R}^m: \alpha |P_\pi(z)| < P_{\pi^\perp}(z)\},$$ where $\pi \oplus \pi^\perp = \mathbb{R}^m$, $\alpha>0$, and $P_\pi(z)$ $\&$ $P_{\pi^\perp}(z)$ are the orthogonal projections of $z\in \mathbb{R}^m$ onto $\pi$ and $\pi^\perp$, respectively. 
\end{dfn}



\begin{dfn}
Given a convex function $\Psi$, we denote its corresponding Monge-Amp\`{e}re measure by $$M_\Psi(B):= \mathcal{L}^n(\partial \Psi(B)),$$ where $B \subset \mathbb{R}^n$ is an arbitrary Borel set and $\partial \Psi$ is the sub-differential of $\Psi$.    
\end{dfn}

\begin{dfn} For a convex body $Z$, $t\cdot Z$ denotes the dilation of $Z$ around its barycenter z (center of mass with respect to Lebesgue measure) by a factor $t\geq0$: $$t\cdot Z:=(1-t)z+tZ.$$ 
\end{dfn}

\begin{dfn} \label{daff} A Radon measure $\mu$ on $\mathbb{R}^n$ doubles affinely on $X \subset \mathbb{R}^n$ if there exists $C>0$ such that each point $x\in X$ has a neighborhood $N_x \subset \mathbb{R}^n$ such that each convex body $Z \subset N_x$ with barycenter in $X$ satisfies $\mu[Z] \leq C \mu\large[\frac{1}{2} \cdot Z\large].$
The constant $C$ is called the doubling constant of $\mu$ on $X$, and $N_x$ is referred to as the doubling neighborhood of $\mu$ around $x$.   
\end{dfn}

\begin{dfn} \label{as} Given $\epsilon>0$ and a convex function $\Psi$, we will denote the $\epsilon$ centered affine section of $\Psi$ at a locally convex point $z \in$ dom $\Psi$ (i.e.  the domain of $\Psi$) by $$Z_\epsilon(z):=Z_\epsilon^\Psi (z)=\{x\in \mathbb{R}^n: \Psi(x)<\epsilon+\Psi(z)+\langle \nu_\epsilon, x-z\rangle \},$$ where $\nu_\epsilon \in \mathbb{R}^n$ is uniquely chosen so that $z$ is the barycenter of $Z_\epsilon(z)$ (see \cite[Theorem A.7 and Lemma A.8]{CM} and \cite{C2}).  
\end{dfn}

\begin{dfn} \label{puc} Fix $p\geq 2$ and a domain $\Omega \subset \mathbb{R}^n$. A locally Lipschitz function $\Psi: \Omega \rightarrow \mathbb{R}$ is $p$-uniformly convex on $\Omega$ if there exists $C>0$ such that all points of differentiability $x,x' \in \Omega \cap dom \nabla \Psi$ satisfy $$\langle \nabla \Psi(x)-\nabla \Psi (x'), x-x' \rangle \geq C|x-x'|^p,$$ where $dom \nabla \Psi$ is the domain of $\nabla \Psi$. 
\end{dfn}

\noindent For an arbitrary convex function $\Psi$, we recall that its \textit{Legendre transform} is the convex function 
\begin{equation} \label{lege}
\Psi^*(y):=\sup_{x \in \mathbb{R}^n}\big( x\cdot y - \Psi(x)\big).
\end{equation}

\begin{rem} \label{hold}
As mentioned in \cite[Remark 7.10]{CM}, if a convex function $\Psi$ is $p$-uniformly convex on $\Omega \subset dom \Psi$, then $\Psi^* \in C^{1,\frac{1}{p-1}}(\partial \Psi(\Omega))$. 
\end{rem}

\begin{dfn} \label{exp}
Let $Z \subset \mathbb{R}^n$ be a closed convex set. A point $p \in Z$ is said to be \textit{exposed} if some hyperplane touches $Z$ only at $p$.  
\end{dfn}

\begin{dfn} \label{ext}
Let $Z \subset \mathbb{R}^n$ be a closed convex set. A point $p \in Z$ is said to be \textit{extreme} if whenever $p=(1-\lambda)p_0+\lambda p_1$ with $\lambda \in (0,1)$, then $p_0=p_1$.    
\end{dfn}

\subsection{Setup}
Given two non-negative, compactly supported functions $f, g \in L^1(\mathbb{R}^n)$, we let $$\Omega:=\{f>0\} \hskip .1in \operatorname{and} \hskip.1in   \Lambda:=\{g>0\},$$ so that $\Omega \cap \Lambda = \{f\wedge g>0\}$. We will always assume $m$ to satisfy: $$||f \wedge g||_{L^1(\mathbb{R}^n)}\leq m \leq \min\{||f||_{L^1(\mathbb{R}^n)}, ||g||_{L^1(\mathbb{R}^n)}\} \large.$$ By the results of Figalli \cite[Section 2]{Fi}, we know that there exists a convex function $\Psi_m$ and non-negative functions $f_m$,  $g_m$ for which $$\gamma_m:=(Id \times \nabla \Psi_m)_{\#}f_m=(\nabla \Psi_m^* \times Id)_{\#}g_m,$$ is the solution of (\ref{mini}) and $\nabla {\Psi_m}_{\#}f_m=g_m$ (see \cite[Theorem 2.3]{Fi}). 

Figalli refers to $\Psi_m$ as the \textit{Brenier solution} to the Monge-Amp\`{e}re equation $$\operatorname{det}(D^2\Psi_m)(x)=\frac{f_m(x)}{g_m(\nabla \Psi_m(x))},$$ with $x \in F_m:=$ set of density points of $\{f_m>0\},$ and $\nabla \Psi_m(F_m) \subset G_m$:= set of density points of $\{g_m>0\}.$ Moreover, following Figalli \cite[Remark 3.2]{Fi}, we set $$U_m:=(\Omega \cap \Lambda) \cup \bigcup_{(\bar x, \bar y) \in \Gamma_m} B_{|\bar x - \bar y|}(\bar y),$$

$$V_m:=(\Omega \cap \Lambda) \cup \bigcup_{(\bar x, \bar y) \in \Gamma_m} B_{|\bar x - \bar y|}(\bar x),$$ where $\Gamma_m$ is the set $$(Id \times \nabla \Psi_m)(F_m \cap D_{\nabla \Psi_m}) \cap (\nabla \Psi_m^* \times Id)(G_m \cap D_{\nabla \Psi_m^*}),$$ with $D_{\nabla \Psi_m}$ and $D_{\nabla \Psi_m^*}$ denoting the set of continuity points for $\nabla \Psi_m$ and $\nabla \Psi_m^*$, respectively.

We denote the free boundary associated to $f_m$ by $\overline{\partial U_m \cap \Omega}$ and the free boundary associated to $g_m$ by $\overline{\partial V_m \cap \Lambda}$. They correspond to $\overline{\partial F_m \cap \Omega}$ and $\overline{\partial G_m \cap \Lambda}$, respectively \cite[Remark 3.3]{Fi}. Recall from the introduction that one of the goals in this paper is to study the regularity of the free boundaries away from $\partial(\Omega \cap \Lambda)$. One method of attacking this problem is to first prove regularity results on $\Psi_m$ and then utilize that $\nabla \Psi_m$ gives the direction of the normal to the free boundary $\overline{\partial U_m \cap \Omega}$ \big(by symmetry and duality, this would also imply a similar result for $\overline{\partial V_m \cap \Lambda}$\big). Indeed, in the following two theorems, Figalli employs this strategy to obtain local $C^1$ regularity.

\begin{thm} \label{Fii1} \cite[Theorem 4.10]{Fi} Suppose $f,g$ are supported on two bounded, open, strictly convex sets $\Omega \subset \mathbb{R}^n$ and $\Lambda \subset \mathbb{R}^n$ respectively, and $$||\operatorname{log}(f(x)/g(y))||_{L^\infty(\Omega \times \Lambda)}<\infty.$$ Then there exists a convex function $\tilde \Psi_m \in C^1(\mathbb{R}^n)\cap C_{loc}^{1,\alpha}(U_m \cap \Omega)$ such that $\tilde \Psi_m = \Psi_m$ on $U_m \cap \Omega$, $\nabla \tilde \Psi_m(x)=x$ on $\Lambda \setminus \overline V_m$, and $\nabla \tilde \Psi_m(\mathbb{R}^n)=\overline \Lambda$. Moreover, $\nabla \tilde \Psi_m : \overline{U_m \cap \Omega} \rightarrow \overline{V_m \cap \Lambda}$ is a homeomorphism (with inverse $\nabla \tilde \Psi_m^*$). 
\end{thm}

\begin{thm} \label{Fii2} \cite[Theorem 4.11]{Fi} Assume the setup in Theorem \ref{Fii1}. Then $(\partial U_m \cap \Omega) \setminus \partial \Lambda$ is locally a $C^1$ surface, and the vector $\nabla \tilde \Psi_m(x)-x$ is different from zero, and gives the direction of the inward normal to $U_m$.  
\end{thm} 

\begin{rem} \label{c1rem}
If $x \in (\partial U_m \cap \partial \Omega) \setminus \partial (\Omega \cap \Lambda)$, then $\nabla \tilde \Psi_m(x) \neq x$ and the same argument used to prove Theorem \ref{Fii2} shows that $\overline{\partial U_m \cap \Omega}$ is locally $C^1$ away from $\partial(\Omega \cap \Lambda)$.  
\end{rem}

In our study, we shall also make frequent use of the fact that free boundary never maps to free boundary. This is summed up in the following proposition \cite[Proposition 4.13]{Fi}:

\begin{prop} (Free boundary never maps to free boundary) \label{fbd}
Assume the setup in Theorem \ref{Fii1} and let $\tilde \Psi_m$ be the corresponding extension of $\Psi_m$. Then\\ 
(a) if $x \in \partial U_m \cap \Omega$, then $\nabla \tilde \Psi_m(x) \notin \overline{\partial V_m \cap \Lambda}$;\\
(b) if $x \in \partial U_m \cap \partial \Omega$, then $\nabla \tilde \Psi_m(x) \notin \partial V_m \cap \Lambda$. 
\end{prop}

\noindent Moreover, we will also need the fact that the common region $\Omega \cap \Lambda$ is contained in the active regions \cite[Remark 3.4]{Fi}:

\begin{rem} \label{cmr}
$$\Omega \cap \Lambda \subset U_m \cap \Omega, \hskip .1in \Omega \cap \Lambda \subset V_m \cap \Lambda.$$ 
\end{rem}

In order to analyze the singular set for the free boundaries, we recall two more sets which will play a crucial role in the subsequent analysis; cf. \cite[Equations (7.1) and (7.2)]{CM}. The nonconvex part of the free boundary $\overline{\partial U_m \cap \Omega}$ is the closed set 
\begin{equation} \label{nc}
\partial_{nc} U_m:=\{x \in \overline{\Omega \cap U_m}: \Omega \cap U_m \hskip.1in \text{fails to be locally convex at $x$}\}.
\end{equation}  
Moreover, the nontransverse intersection points are defined by 

\begin{equation} \label{nt}
\partial_{nt} \Omega:= \{x \in \partial \Omega \cap \overline{\Omega \cap \partial U_m}: \langle \nabla \tilde \Psi_m(x)-x, z-x\rangle \leq 0 \hskip .1in \forall z \in \Omega \}.
\end{equation}
By duality, $\partial_{nc} V_m$ and $\partial_{nt} \Lambda$ are similarly defined. 
\subsection{Tools}
Next, we collect several well-known results from the literature of convex analysis and geometric measure theory which will be useful in our subsequent analysis. The following lemma is a slight adaptation of such a result \cite[Proposition 10.9]{Ma}. Its corollary follows by a standard covering argument.  

\begin{lem} \label{geom} 
Let $M \subset \mathbb{R}^m$ be compact and suppose $\pi$ is an $(m-1)$-dimensional hyperplane. If there exist $\delta>0$ and $\alpha>0$ such that for all $x \in M$, $$(B_{\delta}(x) \cap M) \cap (x+C_\alpha(\pi)) = \emptyset,$$ then there exist $N \in \mathbb{N}$ and Lipschitz functions $f_i: \mathbb{R}^{m-1} \rightarrow \mathbb{R}^m$ where $i \in \{1,\ldots,N\}$, such that $M=\bigcup_{i=1}^N f_i(K_i)$, with $K_i \subset \mathbb{R}^{m-1}$ compact. In particular, $H^{m-1}(M) < \infty$. 
\end{lem}

\begin{cor} \label{corgeom}
Let $M \subset \mathbb{R}^m$ be compact and suppose that for each $x \in M$, $\pi(x)$ is an $(m-1)$-dimensional hyperplane. If there exist $\delta>0$ and $\alpha>0$ such that for all $x \in M$, $$(B_{\delta}(x) \cap M) \cap (x+C_\alpha(\pi(x))) = \emptyset,$$ then there exist $D \in \mathbb{N}$ and Lipschitz functions $f_i: \mathbb{R}^{m-1} \rightarrow \mathbb{R}^m$ where $i \in \{1,\ldots,D\}$, such that $M=\bigcup_{i=1}^D f_i(K_i)$, with $K_i \subset \mathbb{R}^{m-1}$ compact. In particular, $H^{m-1}(M) < \infty$.    
\end{cor}

The next Lemma quantifies the geometric decay of the sections of an arbitrary convex function whose Monge-Amp\`{e}re measure satisfies a doubling property (see Definition \ref{as}). The proof may be found in Caffarelli and McCann \cite[Lemma 7.6]{CM}.

\begin{lem} \label{c}
Given $0\leq t<\bar t \leq 1$ and $C>0$, there exists $s_0=s_0(t,\bar t, \delta,n) \in (0,1)$, such that whenever $Z_{\epsilon}$ is a fixed section centered at $z_0\in$ $X:=spt M_{\Psi}$ of a convex function $\Psi:\mathbb{R}^n \rightarrow (-\infty, \infty]$ whose Monge-Amp\`{e}re measure satisfies the doubling condition 
$$M_{\Psi}[Z_{s\epsilon}(z)]\leq C M_{\Psi}[\frac{1}{2} \cdot Z_{s\epsilon}(z)]$$ for all $s\in [0,1]$ and all $z$ in the convex set $X \cap Z_{\epsilon}(z_0)$, then 
$$z \in X \cap t\cdot Z_{\epsilon}(z_0) \implies Z_{s\epsilon}(z) \subset \bar t \cdot Z_{\epsilon}(z_0), \hskip .3in \forall s\leq s_0.$$   
\end{lem}
 
\begin{cor} \label{core1}
Assuming the setup in Lemma \ref{c}, we have $$Z_{s^k\epsilon}(x) \subset \bar t^k \cdot Z_{\epsilon}(x),$$ for all $s<s_0(0,\bar t)$, $\bar t \in (0,1)$ and integers $k \geq0$.  
\end{cor}

\noindent The following theorem of Straszewicz establishes a connection between exposed and extreme points of a closed convex set \cite[Theorem 18.6]{Ro}. 

\begin{thm} \label{Str}
For any closed convex set $Z\subset \mathbb{R}^n$, the set of exposed points of $Z$ is a dense subset of the set of extreme points of $Z$. 
\end{thm}

\noindent The next theorem of Blaschke is a classical result which states that a family of convex bodies living in a ball admits a converging subsequence in the Hausdorff topology \cite{Bl}.  
\begin{thm} \label{Bl}
The space of all convex bodies in $\mathbb{R}^n$ is locally compact with respect to the Hausdorff metric. 
\end{thm}

\section{The $C^{1,\alpha}_{loc}$ regularity theory}
\label{Section3}

In what follows, we apply the method of Caffarelli $\&$ McCann \cite[Section 7]{CM} to derive the $C^{1,\alpha}_{loc}$ regularity of the free boundary away from the common region. Unless otherwise stated, we will always assume the following on the initial data:\\

\noindent \textbf{Assumption 1:} Assume $f,g$ are bounded away from zero and infinity on strictly convex, bounded domains $\Omega \subset \mathbb{R}^n$ and $\Lambda \subset \mathbb{R}^n$, respectively.\\

\noindent Indeed, this is the main assumption of Theorem \ref{Fii1}, therefore, whenever we will employ this theorem in the statements of our results, Assumption $1$ will be implicit. We start the analysis by identifying the support of the Monge-Amp\`{e}re measure corresponding to $\tilde\Psi_m$. By using an equation from the work of Figalli \cite[Equation (4.5)]{Fi}, one may prove this result in a similar manner (in fact, almost verbatim) as was done in Caffarelli and McCann \cite[Lemma 7.2]{CM}; hence, we omit the details.
\begin{lem} \label{a}
Let $\tilde\Psi_m$ be the $C^1(\mathbb{R}^n)$ extension of $\Psi_m$ to $\mathbb{R}^n$ given by Theorem \ref{Fii1}. Then $\tilde\Psi_m$ has a Monge-Amp\`{e}re measure that is absolutely continuous with respect to Lebesgue, and there exist positive constants $c$, $C$ (depending on the initial data) so that for any Borel set $E \subset \mathbb{R}^n$,
\begin{equation}
c |E\cap (\Omega \cap U_m)| + |E\cap (\Lambda \setminus V_m)|\leq M_{\tilde\Psi_m}(E) \leq C|E\cap (\Omega \cap U_m)|+|E \cap (\Lambda \setminus V_m)|.  
\end{equation} \label{a1}
\end{lem}

Next, we identify a set on which the Monge-Amp\`{e}re measure corresponding to the convex function $\tilde\Psi_m$ doubles affinely (recall Definition \ref{daff}). This will be useful in quantifying the strict convexity of $\tilde\Psi_m$. 

\begin{lem} \label{b}
Let $\tilde\Psi_m$ be the $C^1(\mathbb{R}^n)$ extension of $\Psi_m$ to $\mathbb{R}^n$ given by Theorem \ref{Fii1}. Then $\tilde\Psi_m$ has a Monge-Amp\`{e}re measure $M_{\tilde \Psi_m}$ which doubles affinely (see Definition \ref{daff}) on $$X:=\overline{\Omega \cap U_m} \setminus \Big(\partial_{nc} U_m \cup \big(\overline{\partial V_m \cap \Lambda} \cap \partial (\Omega\cap \Lambda)\big)\Big).$$ Moreover, any ball $N_x=B_R(x)$ which has a convex intersection with $\Omega \cap U_m$ and is disjoint from $\Lambda \setminus V_m$  is a doubling neighborhood around $x$.   
\end{lem}

\begin{proof}
First, since $\partial V_m \cap \Lambda$ does not intersect $U_m\cap \Omega$ (by Remark \ref{cmr}), the only place where $\Lambda \setminus V_m$ may possibly intersect $\overline{U_m \cap \Omega}$ is on $\overline{\partial V_m \cap \Lambda} \cap \partial (\Omega \cap \Lambda)$. Now if $x \in X$, then there exists $R=R(x)>0$ for which $B_R(x) \cap (\Lambda \setminus V_m) = \emptyset$ and $W:= \Omega \cap U_m \cap B_R(x)$ is convex. With this in mind, thanks to Lemma \ref{a}, we may proceed verbatim as in \cite[Lemma 7.5]{CM} and \cite[Lemma 2.3]{C2} to prove the result. 

\end{proof}

Note that the set $X$ from the previous lemma on which $M_{\tilde \Psi_m}$ doubles affinely excludes non-locally convex points in $\overline{\Omega \cap U_m}$; since Caffarelli's regularity theory employs the doubling property, and since the active region is not necessarily convex, this suggests the existence of a potential singular set. Indeed, we will now define and prove some topological results of various sets which will naturally come up in the course of our study; these sets will be used to construct candidates for the singular set. Although seemingly technical, they have a very geometric flavor, see Figure \ref{xenon}.

\begin{dfn} \label{d1} (Components of the singular set)
Let $\tilde \Psi_m$ be as in Theorem \ref{Fii1}. Then, for $x\in \partial (\Omega \cap U_m)$ and $x \neq \nabla \tilde \Psi_m(x)$, let\\
\vskip .1in 
\noindent $L(x):=\Bigl\{\nabla \tilde \Psi_m(x)+ \frac{x-\nabla \tilde \Psi_m(x)}{|x-\nabla \tilde \Psi_m(x)|}t: t\geq0 \Bigr\};$\\
\noindent
\vskip .1in
\noindent $K:=\Bigl\{x \in \partial (\Omega \cap U_m): \nabla \tilde \Psi_m(x)\neq x,  L(x) \cap \overline{\Omega \cap U_m} \subset \partial (\overline{\Omega \cap U_m}) \Bigr\};$\\
\vskip .1in
\noindent The following two sets play a critical role in our study:\\ 
\vskip .1in
\noindent $S_1:=\nabla \tilde \Psi_m^{-1}(\partial_{nt} \Lambda) \cap K;$   \\
\vskip .1in  
\noindent $S_2:=\big (\partial U_m \cap \partial \Lambda \cap \Omega\big) \cup \big(\partial \Omega \cap \partial \Lambda \cap \{\nabla \tilde \Psi_m(z)=z\}\big) \cup (\partial V_m \cap \partial \Omega \cap \Lambda).$\\
\vskip .1in
\noindent It will prove useful for us to decompose $S_1$ into the part which touches the free boundary and the part which is disjoint from the free boundary: 
\vskip .1in 
\noindent $A_1:=S_1 \cap \partial U_m$; \\

\vskip .1in 
\noindent $A_2:=S_1\setminus \partial U_m.$\\  
\end{dfn}
\begin{figure}[h!]
\centering 
\includegraphics[scale= .5]{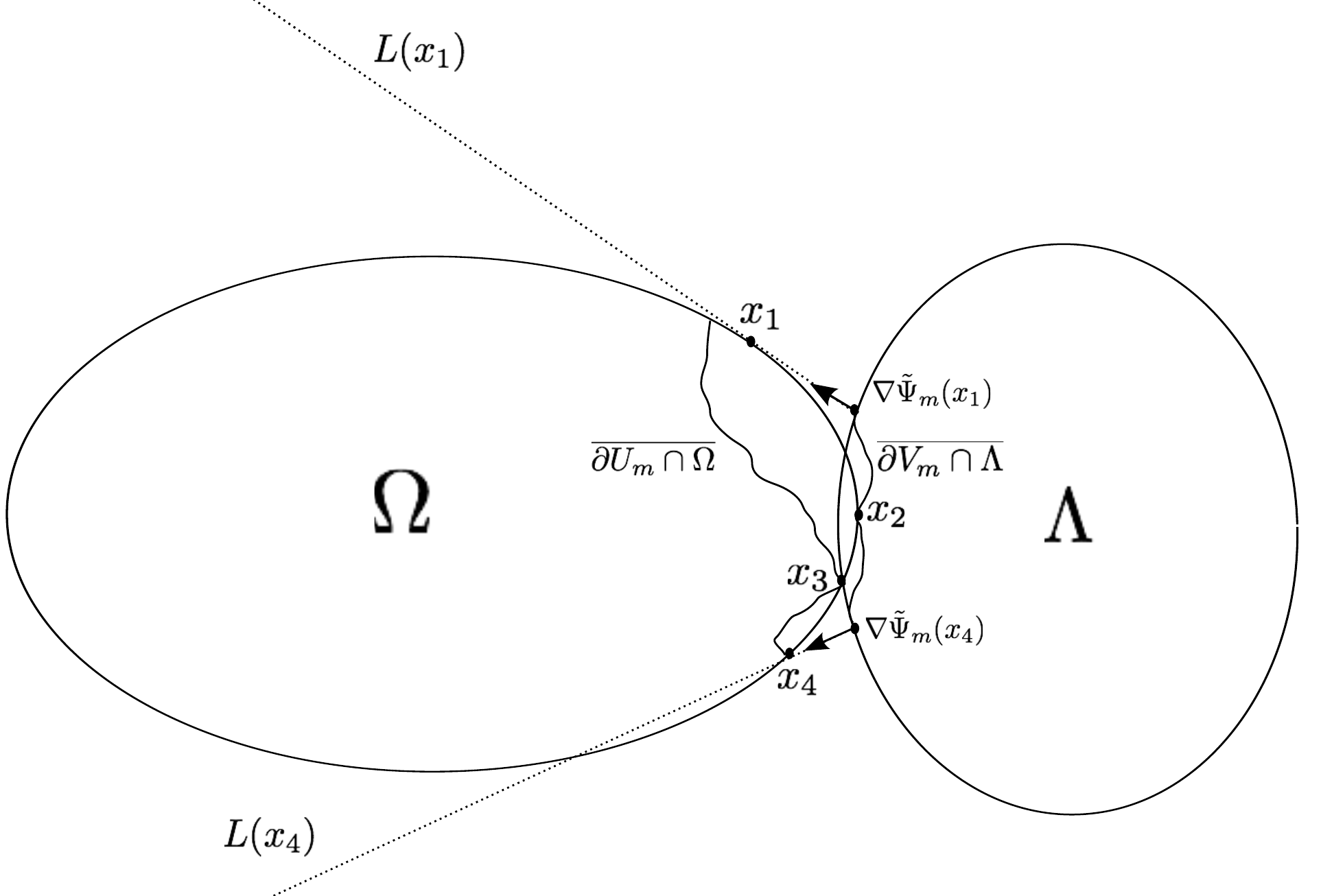}
\caption{$x_1\in A_2$; $x_4 \in A_1$; $x_2,x_3 \in S_2$.}
\label{xenon}
\end{figure}

\begin{rem} \label{r0}
If $x\in S_1 \cap \Omega$, then $x\in \partial U_m \cap \Omega$ and $\nabla \tilde \Psi_m(x) \in \partial_{nt} \Lambda \subset \overline{\partial V_m \cap \Lambda}$, a contradiction to Proposition \ref{fbd}. Hence, $S_1 \subset \partial \Omega$.    
\end{rem}

\begin{lem} \label{s01}
$\big(\partial U_m\cap \partial \Lambda \cap \overline{\Omega} \big) \cup \big(\partial V_m \cap \partial \Omega \cap \overline{\Lambda}\big) \subset \{\nabla \tilde \Psi_m(z)=\nabla \tilde \Psi_m^*(z)=z\} $.
\end{lem}
\begin{proof}
Suppose first that $x \in \partial V_m \cap \partial \Omega \cap \overline \Lambda$. Then, since $\partial V_m \cap \partial \Omega \cap \overline \Lambda \subset \overline{\Lambda \setminus \overline{V_m}}$, it follows that $\nabla \tilde \Psi_m(x)=x$ by Theorem \ref{Fii1}. But by Remark \ref{cmr}, we know $$\Omega \cap \Lambda \subset (\Lambda \cap V_m) \cap (\Omega \cap U_m);$$ therefore, $\partial V_m \cap \partial \Omega \cap \overline \Lambda \subset \overline{\Lambda \cap V_m} \cap \overline{\Omega \cap U_m}$, and since $\nabla \tilde \Psi_m : \overline{\Omega \cap U_m} \rightarrow \overline{\Lambda \cap V_m}$ is a homeomorphism with inverse $\nabla \tilde \Psi_m^*$,  we have $\nabla \tilde \Psi_m^*(x)=\nabla \tilde \Psi_m^*(\nabla \tilde \Psi_m(x))=x$. An entirely symmetric argument yields $\partial U_m\cap \partial \Lambda \cap \overline{\Omega} \subset \{\nabla \tilde \Psi_m(z)=\nabla \tilde \Psi_m^*(z)=z\}.$ 
\end{proof}

\begin{rem} \label{r1}
We note that if $x \in S_2$, then by Lemma \ref{s01}, $\nabla \tilde \Psi_m(x)=x$ so $\nabla \tilde \Psi_m(S_2)=S_2$.   
\end{rem}

\begin{lem} \label{s2}
Let $X_s:=S_1 \cup S_2$. Then $S_2$ and $X_s$ are compact. 
\end{lem}

\begin{proof}
First, we note that $X_s \subset \overline{\Omega \cup \Lambda}$, and since $\overline{\Omega \cup \Lambda}$ is bounded, it suffices to prove that $S_2$ and $X_s$ are closed. First, we prove the assertion for $S_2$: note that by Lemma \ref{s01}, 
\begin{align*}
\overline S_2 &\subset \big (\partial U_m \cap \partial \Lambda \cap \overline{\Omega}\big) \cup \big(\partial \Omega \cap \partial \Lambda \cap \{\nabla \tilde \Psi_m(z)=z\}\big) \cup (\partial V_m \cap \partial \Omega \cap \overline{\Lambda})\\
&\subset \Bigl( (\partial U_m \cap \partial \Lambda \cap \overline{\Omega}\big) \cup (\partial \Omega \cap \partial \Lambda)  \cup (\partial V_m \cap \partial \Omega \cap \overline{\Lambda}) \Bigr)  \cap \{\nabla \tilde \Psi_m(z)=z\}\\
&\subset S_2 \cup \big(\partial \Omega \cap \partial \Lambda \cap \{\nabla \tilde \Psi_m(z)=z\}\big)\subset S_2.  
\end{align*}

\noindent Next, we show $\overline S_1 \subset S_1 \cup S_2 =X_s$. Indeed, suppose $\{x_n\} \subset S_1$ with $x_n \rightarrow x \in \overline{\Omega \cup \Lambda}$. Then, as $\partial(\Omega \cap U_m)$ is compact, we have that 
\begin{equation} \label{s22}
x \in \partial(\Omega \cap U_m).
\end{equation}
Let $y_n:=\nabla \tilde \Psi_m(x_n) \in \partial_{nt} \Lambda \subset \partial \Lambda \cap \partial V_m$ so that for all $z \in \Lambda$, $$\langle \nabla \tilde \Psi_m^*(y_n)-y_n,z-y_n\rangle \leq 0.$$ By continuity of $\nabla \tilde \Psi_m$ and compactness of $\partial \Lambda \cap \partial V_m$, $y_n \rightarrow \nabla \tilde \Psi_m(x)=:y \in \partial \Lambda \cap \partial V_m$, and by continuity of $\nabla \tilde \Psi_m^*$ and of the inner product, it follows that for all $z \in \Lambda$, $$\langle \nabla \tilde \Psi_m^*(y)-y,z-y\rangle \leq 0.$$ Hence, $y \in \partial_{nt} \Lambda$ and 
\begin{equation} \label{s21}
x=\nabla \tilde \Psi_m^*(y)= (\nabla \tilde \Psi_m)^{-1}(y) \in \nabla \tilde \Psi_m^{-1}(\partial_{nt} \Lambda).
\end{equation} 
Let us first assume $\nabla \tilde \Psi_m(x) \neq x$. In this case, if there exists $t \geq 0$ such that $$\nabla \tilde \Psi_m(x)+ \frac{x-\nabla \tilde \Psi_m(x)}{|x-\nabla \tilde \Psi_m(x)|}t \in \Omega \cap U_m,$$ then since $\Omega \cap U_m$ is open, for $n$ large enough we will also have $$\nabla \tilde \Psi_m(x_n)+ \frac{x_n-\nabla \tilde \Psi_m(x_n)}{|x_n-\nabla \tilde \Psi_m(x_n)|}t \in \Omega \cap U_m,$$ a contradiction to the fact that $x_n \in K$. Therefore, we obtain that for all $t \geq 0$, $$\nabla \tilde \Psi_m(x)+ \frac{x-\nabla \tilde \Psi_m(x)}{|x-\nabla \tilde \Psi_m(x)|}t \not \in \Omega \cap U_m;$$ hence, $x \in K$ and together with (\ref{s21}), we obtain $x \in S_1$. Now it may happen that $\nabla \tilde \Psi_m(x) = x$. In this case, by $(\ref{s21})$, we know $x=\nabla \tilde \Psi_m(x) \in \partial_{nt} \Lambda$, so in particular $x \in \partial \Lambda \cap \partial V_m$. Moreover, by (\ref{s22}), we also have $x \in \partial(\Omega \cap U_m)$. If $x \in \partial U_m \cap \Omega$, then it follows that $x \in \partial V_m \cap \Omega$, a contradiction to the fact that the free boundary does not enter the common region (see Remark \ref{cmr}). Therefore, we must have $x \in \partial V_m \cap \partial \Omega$; hence, Lemma \ref{s01} implies $x \in S_2$ and so $\overline S_1 \subset S_1 \cup S_2$.                      
\end{proof}

\begin{cor} \label{impa}
Let $\tilde\Psi_m$ be the $C^1(\mathbb{R}^n)$ extension of $\Psi_m$ to $\mathbb{R}^n$ given by Theorem \ref{Fii1}. Then for all $z \in \partial \Lambda$ and $R>0$ with $\overline{B_R(z)} \cap \partial \Omega = \emptyset$, we have that $\nabla \tilde \Psi_m(S_1) \cap \overline{B_R(z)}$ is compact. In particular, $\nabla \tilde \Psi_m(S_1) \setminus \partial \Omega$ is relatively closed in $\partial \Lambda \setminus \partial \Omega$.    
\end{cor}

\begin{proof}
It suffices to prove that $\nabla \tilde \Psi_m(S_1) \cap \overline{B_R(z)}$ is closed. Let $y_n \in \nabla \tilde \Psi_m(S_1) \cap \overline{B_R(z)}$ and suppose $y_n \rightarrow y \in \partial_{nt} \Lambda \cap \overline{B_R(z)}$. Set $x_n:= \nabla \tilde \Psi_m^*(y_n)$ and $x:=\nabla \tilde \Psi_m^*(y)$. Then by repeating the proof of Lemma \ref{s2}, it follows that $x \in \partial(\Omega \cap U_m)$, $L(x) \cap \overline{\Omega \cap U_m} \subset \partial(\overline{\Omega \cap U_m})$, and $x \in \nabla \tilde \Psi_m^{-1}(\partial_{nt} \Lambda)$. Since $y_n \in \nabla \tilde \Psi_m(S_1)$, it also follows from Remark \ref{r0} that $x_n \in \partial \Omega$; hence, $x\in \partial \Omega$. Now if $y=\nabla \tilde \Psi_m(x)=x$, then $y \in \partial \Omega$. However, $y \in \overline{B_R(z)}$, and by assumption, $\overline{B_R(z)} \cap \partial \Omega = \emptyset$. Thus, $\nabla \tilde \Psi_m(x) \neq x$, and $y \in\nabla \tilde \Psi_m(S_1) \cap \overline{B_R(z)}$. Since $\nabla \tilde \Psi_m(S_1) \cap \overline{B_R(z)}$ is compact, $\nabla \tilde \Psi_m(S_1) \setminus \partial \Omega$ is relatively closed in $\partial \Lambda \setminus \partial \Omega$.

\end{proof}

\begin{rem} \label{bbgun}
By arguing as in the proof of Corollary \ref{impa}, one may similarly deduce that the set $\nabla \tilde \Psi_m(A_1) \setminus \partial \Omega$ is relatively closed in $\partial \Lambda \setminus \partial \Omega$. Moreover, it is not hard to see that $S_1$ and $A_1$ are relatively closed in $\partial \Omega \setminus \partial \Lambda$.  
\end{rem}

\noindent Next, we generalize the uniform localization lemma of Caffarelli and McCann \cite[Lemma 7.11]{CM} to the case in which the supports may have a nontrivial intersection. Our proof is by contradiction and follows the line of reasoning for the disjoint case although a new ingredient is required to get around the lack of a separating hyperplane. Our key observation is that one may fully identify the exposed points of the closed convex set $Z_{min}$ that shows up in the work of Caffarelli and McCann. Indeed in that context, thanks to the fact that $\overline{\Omega}\cap \overline{\Lambda}=\emptyset$, the authors only need that all exposed points lie in the set $\overline{\Lambda}$ to obtain the contradiction; however, this is not enough in the general case. We get around this difficulty by exploiting the fact that $\nabla \tilde \Psi_m(x)=x$ for all $x \in \Lambda \setminus V_m $ (see Theorem \ref{Fii1}). Consequently, a weaker version of the uniform localization lemma for the disjoint case is established; this paves the way for the next section in which we estimate the Hausdorff dimension of the singular set.

\begin{lem} \label{d} (Uniform localization: general case)
Let $\tilde\Psi_m$ be the $C^1(\mathbb{R}^n)$ extension of $\Psi_m$ to $\mathbb{R}^n$ given by Theorem \ref{Fii1} and $X_{s}$ the compact set in Lemma \ref{s2}. Then for $R>0$ there exists $\epsilon_0>0$ such that for all $z \in \overline{\Omega \cap U_m}$ for which $B_R(z) \cap X_{s} = \emptyset$ and for all $0\leq \epsilon \leq \epsilon_0$, we have $$Z_{\epsilon}(z) \subset B_{R}(z).$$
\end{lem}

\begin{proof}
Suppose not. Then there exists $R>0$ such that for all $k\in \mathbb{N}$ there exists $0<\epsilon_k \leq \frac{1}{k}$ and $z_k\in \overline{\Omega \cap U_m}$ satisfying $B_R(z_k)\cap X_s =\emptyset$ and $Z_{\epsilon(k)}(z_k) \not \subset B_R(z_k)$. Since $\overline{\Omega \cap U_m}$ is compact, along a subsequence we have $z_k \rightarrow z_{\infty} \in \overline{\Omega \cap U_m},$ with 
\begin{equation} \label{empt}
B_R(z_{\infty}) \cap X_s =\emptyset.
\end{equation}
By translating all the data we may assume $\nabla \tilde \Psi_m(z_{\infty})=0$. Since $\tilde \Psi_m$ is convex, this implies that $\tilde \Psi_m$ is minimized at $z_{\infty}$. Now by Theorem \ref{Fii1}, $\nabla \tilde \Psi_m(\mathbb{R}^n)=\overline{\Lambda}$ is bounded and each centered affine section is bounded, thus it follows that the slope $\nu_{\epsilon(k)}(z_k)$ of the affine function defining the set $Z_{\epsilon(k)}(z_k)$ is contained in $\Lambda$ (indeed, a translate of the affine function defining the section serves as a supporting hyperplane for $\tilde\Psi_m$). Therefore, along another subsequence $\nu_{\epsilon(k)}(z_k) \rightarrow \nu_{\infty} \in \overline{\Lambda}$ and we can apply Theorem \ref{Bl} (Blaschke selection theorem) to conclude that the sets $Z_{\epsilon(k)}(z_k)$ converge locally in Hausdorff distance to a closed convex set $Z_{\infty}$. Let $Z_{min}:=\{x \in \mathbb{R}^n: \tilde \Psi_m(x)=\tilde \Psi_m(z_{\infty})\}$. By the same exact argument as in \cite[Lemma 7.11 (Claim \#1)]{CM}, one derives $Z_{\infty} \subset Z_{min}$ and that $Z_\infty$ contains a line segment $L$ centered at $z_\infty$ of length $\frac{2R}{\alpha},$ where $\alpha:=n^{\frac{3}{2}}$ is the constant from John's Lemma (the idea is that if strict convexity fails at a point, then there must be a segment on which the function is affine). Now by Theorem \ref{Fii1}, we know $\nabla \tilde \Psi_m : \overline{U_m \cap \Omega} \rightarrow \overline{V_m \cap \Lambda}$ is a homeomorphism; hence, $Z_{min}$ cannot intersect $\overline{U_m \cap \Omega}$ except at the single point $z_{\infty}$, which necessarily, must lie on the boundary. Therefore, the set $Z_{min} \setminus \{z_{\infty}\}$ must lie outside of $\overline{U_m \cap \Omega}$. Next, by the same exact argument as in \cite[Lemma 7.11 (Claim \#2)]{CM} we have that the exposed points of $Z_{min}$ (see Definition \ref{exp}) lie in the support of the Monge-Amp\`{e}re measure of $\tilde \Psi_m$. By Lemma \ref{a}, this implies that the exposed points of $Z_{min}$ lie in $\overline{\Omega \cap U_m}$ or $\overline{\Lambda \setminus V_m}$; since $\{z_{\infty}\}=Z_{min} \cap \overline{\Omega \cap U_m}$ and $z_\infty$ is not an exposed point in $Z_{min}$ (due to the existence of $L$), we have that all exposed points of $Z_{min}$ lie in $\overline{\Lambda \setminus V_m}$. Since every extreme point (see Definition \ref{ext}) is a limit of exposed points (by Theorem \ref{Str}), we have that the extreme points of $Z_{min}$ also lie in $\overline{\Lambda \setminus V_m}$. Next, note that if $Z_{min}$ would contain a whole line, then gradient monotonicity would imply $\nabla \tilde \Psi_m \cdot e_1 = 0$, where $e_1$ is the direction of the line. This however, contradicts $\nabla \tilde \Psi_m(\mathbb{R}^n) = \overline{\Lambda}$. Since the closed, convex set $Z_{min}$ does not contain a line, by Minkowski's theorem \cite[Theorem 18.5]{Ro} we have $Z_{min}=conv[ext[Z_{min}]+rc[Z_{min}]]$. Also, since $z_{\infty} \in Z_{min}$, we have $Z_{min}$ is non-empty, so $ext[Z_{min}]$ is non-empty.  Hence, $z_{\infty} = \sum t_i (x_i+y_i)$, where $\sum t_i = 1$, $x_i \in \overline{\Lambda \setminus V_m}$, and $y_i \in rc[Z_{min}]$. Since the recession cone of a convex set is convex, we have that $y:=\sum t_i y_i \in rc[Z_{min}]$. Moreover, by Theorem \ref{Fii1}, $\nabla \tilde \Psi_m(x)=x$ on $\Lambda \setminus \overline{V_m}$ and by continuity on $\Lambda \setminus V_m$. Combining this fact with the definition of $Z_{min}$ yields $0=\nabla \tilde \Psi_m(x_i)=x_i$. Note that this also shows $0$ to be the only extreme point of $Z_{min}$, which in turn, implies $z_{\infty}=y \in rc[Z_{min}]$. Next, we wish to show 
\begin{equation} \label{z1}
z_{\infty} \neq 0.
\end{equation} 
Assume by contradiction that $z_{\infty}=0$. Recall $\nabla \tilde \Psi_m(z_{\infty})=0$, so in particular $\nabla \tilde \Psi_m(z_{\infty})=z_{\infty}$. However, $z_{\infty} \in \partial(\overline{\Omega \cap U_m})$, and $\nabla \tilde \Psi_m : \overline{U_m\cap \Omega} \rightarrow \overline{V_m\cap \Lambda}$ is a homeomorphism; therefore, $z_{\infty} \in \partial(\overline{\Omega \cap U_m}) \cap \partial (\overline{\Lambda \cap V_m})$. Hence, $$z_{\infty} \in \big((\partial U_m \cap \partial \Lambda \cap \Omega) \cup (\partial \Omega \cap \partial \Lambda) \cup (\partial V_m \cap \partial \Omega \cap \Lambda)\big) \cap \{\nabla \tilde \Psi_m (z)=z\}=S_2,$$ and this contradicts $B_R(z_{\infty})\cap X_s =\emptyset$.

\noindent Thus, since $z_{\infty} \neq 0$ is in $rc[Z_{min}]$, we have that $Z_{min}$ contains a half-line in the direction $z_{\infty}$. Consider a basis for $\mathbb{R}^n$ so that $z_{\infty}$ parallels the negative $x_n$ axis. Let $z \in \mathbb{R}^n$ be arbitrary. Since $\tilde \Psi_m$ is convex, $$\bigg \langle \nabla \tilde \Psi_m(z)- \nabla \tilde \Psi_m(ke_n), \frac{z-ke_n}{|z-ke_n|} \bigg \rangle \geq 0.$$ By taking the limit as $k \rightarrow \infty$, we obtain $\partial_n \tilde \Psi_m(z) \geq 0$. Therefore, it follows that $\nabla \tilde \Psi_m(\mathbb{R}^n) = \overline{\Lambda} \subset \{x: x_n \geq 0\}$. This implies 
\begin{equation} \label{l1}
0 \in \partial \Lambda,
\end{equation}
and that $\{x_n=0\}$ is a supporting hyperplane for $\Lambda$ at $0$. In particular, $z_{\infty}$ is a normal to $\Lambda$ at $0$. Recall that all the extreme points of $Z_{min}$ lie in $\overline{\Lambda \setminus V_m}$ and since $0$ is the only extreme point of $Z_{min}$, 
\begin{equation} \label{l2}
0 \in \overline{\Lambda \setminus V_m}. 
\end{equation}
However, $z_{\infty} \in \partial(\overline{\Omega \cap U_m})$ so $$0=\nabla \tilde \Psi_m(z_{\infty}) \in \partial (\overline{V_m\cap \Lambda}) = (\partial V_m\cap \Lambda) \cup (\partial V_m \cap \partial \Lambda)\cup \Big(\partial \Lambda \setminus \partial \big(\overline{\Lambda \setminus V_m}\big)\Big).$$ Hence, (\ref{l1}) and (\ref{l2}) imply $0 \in \partial V_m \cap \partial \Lambda$; in particular, $0$ is a free boundary point. Since $\nabla \tilde \Psi_m^*(0)=z_{\infty}$ and $z_\infty$ is a normal to $\Lambda$ at $0$, convexity of $\Lambda$ implies $$\langle \nabla \tilde \Psi_m^*(0) - 0, y-0 \rangle \leq 0 \hskip .1in \forall y \in \Lambda;$$ therefore, 
\begin{equation} \label{l3}
z_\infty \in \nabla \tilde \Psi_m^{-1}(\partial_{nt} \Lambda).
\end{equation}
Recall also that $z_\infty \in rc[Z_{min}]$ and $0 \in Z_{min}$ so that in particular, by definition of recession cone, $0+\frac{z_\infty}{|z_\infty|}t \in Z_{min} \hskip .1in \forall t \geq 0,$ which is, of course, equivalent to
\begin{equation} \label{l4} 
\nabla \tilde \Psi_m(z_\infty)+\frac{z_\infty-\nabla \tilde \Psi_m(z_\infty)}{|z_\infty-\nabla \tilde \Psi_m(z_\infty)|}t \in Z_{min} \hskip .1in \forall t \geq 0.
\end{equation}
Since $Z_{min} \cap \overline{\Omega \cap U_m}=\{z_{\infty} \}$, (\ref{z1}), (\ref{l3}), and (\ref{l4}) imply $z_\infty \in S_1$. This contradicts that $z_\infty \notin X_s$ (i.e. (\ref{empt})).  

\end{proof}

\begin{rem} \label{rem111} (Uniform localization: disjoint case)
If $\Omega$ and $\Lambda$ are separated by a hyperplane and $\tilde\Psi_m$ is the $C^1(\mathbb{R}^n)$ extension of $\Psi_m$ to $\mathbb{R}^n$ given by Theorem \ref{Fii1}, then $S_2=\emptyset$ so that $$X_{s}=S_1 \cup S_2 =\nabla \tilde \Psi_m^{-1}(\partial_{nt} \Lambda) \cap K \subset X_{nt}:= \overline{\Omega \cap U_m} \cap \nabla \tilde \Psi_m^{-1}(\partial_{nt} \Lambda).$$ Therefore, we obtain Caffarelli and McCann's uniform localization lemma   
\cite[Lemma 7.11]{CM} under a weaker hypothesis: namely, that of replacing $X_{nt}$ by $X_s$. 
\end{rem}

Equipped with the general uniform localization lemma and the other tools developed so far, we are now in a position to prove that away from a singular set, $\tilde \Psi_m$ will be locally $p$-uniformly convex (recall Definition \ref{puc}); this in turn will readily yield the H\"{o}lder continuity of $\nabla \tilde \Psi_m^*$ (see Remark \ref{hold} and Corollary \ref{cor1}). The proof is a direct adaptation of the corresponding proof for the disjoint case (cf. \cite[Theorem 7.13]{CM}); nevertheless, we have decided to include it in the appendix for the reader's convenience.     

\begin{thm} \label{th1} 
Let $\tilde\Psi_m$ be the $C^1(\mathbb{R}^n)$ extension of $\Psi_m$ to $\mathbb{R}^n$ given by Theorem \ref{Fii1}. Given $x \in \overline{\Omega \cap U_m}$ and  $R>0$ there exists $r=r(R,\epsilon_0)>0$ (where $\epsilon_0$ is from Lemma \ref{d}) such that  $\tilde \Psi_m$ will be p-uniformly convex on $\Omega \cap U_m \cap B_{\frac{r}{2}}(x)$ if $B_{3R}(x)$ is disjoint from the closed set $\overline{\Lambda \setminus V_m}  \cup X_s$ and has convex intersection with $\Omega \cap U_m$.   
\end{thm}

\begin{cor} \label{cor1}
Let $\tilde\Psi_m$ be the $C^1(\mathbb{R}^n)$ extension of $\Psi_m$ to $\mathbb{R}^n$ given by Theorem \ref{Fii1}. Consider the closed set $$F:=\nabla \tilde \Psi_m(\partial_{nc}U_m) \cup \nabla \tilde \Psi_m(X_s),$$ with $X_s$ as in Lemma \ref{s2}. Then $\tilde\Psi_m^* \in C_{loc}^{1,\alpha}\bigl(\overline{\Lambda \cap V_m} \setminus F\bigr)$, where $\alpha:=\frac{1}{p-1}$ and $p$ is as in Theorem \ref{th1}.   
\end{cor}

\begin{proof}
Let $y \in \overline{\Lambda \cap V_m} \setminus F$ and set $x:=\nabla \tilde \Psi_m^*(y) \in \overline{\Omega \cap U_m}$. Note that $x \notin \partial_{nc}U_m$, so there exists $\delta_1=\delta(x)>0$ such that $B_{\delta_1}(x) \cap (\Omega \cap U_m)$ is convex. Moreover, by Lemma \ref{s2}, $X_s:=S_1 \cup S_2$ is compact, and since $x \notin X_s$, there exists $\delta_2>0$ such that $B_{\delta_2}(x) \cap X_s = \emptyset$. Let $\delta:= \min\{\delta_1, \delta_2\}$. Note that since $B_{\delta_2}(x) \cap X_s = \emptyset$ we have $x \notin \partial (\Omega \cap \Lambda) \cap \partial V_m$; thus, by possibly taking $\delta$ smaller we may assume without loss of generality that $B_{\delta}(x) \cap \overline{\Lambda \setminus V_m} = \emptyset$. Then set $R:=\frac{\delta}{3}$ so that by Theorem \ref{th1}, there exists $r=r(R,\epsilon_0)$ (where $\epsilon_0$ is from Lemma \ref{d}) such that  $\tilde \Psi_m$ will be p-uniformly convex on $\Omega \cap U_m \cap B_{\frac{r}{2}}(x)$. Since the convexity exponent and constant are universal, it follows that $\tilde \Psi_m^* \in C^{1,\frac{1}{p-1}}\Bigl(\overline{\nabla \tilde \Psi_m(\Omega \cap U_m \cap B_{\frac{r}{2}}(x))}\Bigr)$. Now $\nabla \tilde \Psi_m(\Omega \cap U_m \cap B_{\frac{r}{2}}(x))$ is relatively open in $\Lambda \cap V_m$ (since $\nabla \tilde \Psi_m$ is a homeomorphism), so there exists $s>0$ such that $\tilde \Psi_m^* \in C^{1,\frac{1}{p-1}}\big(\overline{B_s(y) \cap (\Lambda \cap V_m)}\big)$.             
\end{proof}

\begin{rem} \label{frem}
If $\overline{\Omega} \cap \overline{\Lambda} = \emptyset$, then $\nabla \tilde \Psi_m(S_2)=\emptyset$ and so $F \subset \nabla \tilde \Psi_m(\partial_{nc}U_m) \cup \partial_{nt} \Lambda;$ in particular, we obtain \cite[Corollary 7.14]{CM}.    
\end{rem}

\noindent For $x \in \partial U_m \cap \Omega$, we know that $\nabla \tilde \Psi_m(x)-x$ is parallel to the normal of the free boundary by Theorem \ref{Fii2}. Combining this fact with Corollary \ref{cor1} enables us to derive $C_{loc}^{1,\alpha}$ regularity of the free boundaries inside the domains.    

\begin{cor} \label{cor2} (Free boundary regularity inside the domains)
The free boundaries $\partial V_m \cap \Lambda$ and $\partial U_m \cap \Omega$ are $C_{loc}^{1,\alpha}$ hypersurfaces away from $\partial{(\Omega \cap \Lambda)}$ with $\alpha:= \frac{1}{p-1}$ and $p$ as in Theorem \ref{th1}.   
\end{cor}

\begin{proof}
We prove the result only for $\partial V_m \cap \Lambda$ since the argument for $\partial U_m \cap \Omega$ is entirely symmetric. Let $y \in (\partial V_m \cap \Lambda) \setminus \partial (\Omega \cap \Lambda)$; in particular, $y \notin S_2 = \nabla \tilde \Psi_m(S_2)$ (see Remark \ref{r1}). Moreover, since $\nabla \tilde \Psi_m(S_1) \subset \partial V_m \cap \partial \Lambda$, we also have that $y \notin \nabla \tilde \Psi_m(S_1)$. Next, as $y\in \partial V_m \cap \Lambda$, we may apply Proposition \ref{fbd} (free boundary never maps to free boundary) to deduce $x:=\nabla \tilde \Psi_m^{-1}(y)=\nabla \tilde \Psi_m^{*}(y) \notin \overline{\partial U_m \cap \Omega}$. Therefore, $x \in \partial \Omega \setminus \partial V_m$ and so $y \notin \nabla \tilde \Psi_m(\partial_{nc}U_m)$. Hence, $y \notin F$ and Corollary \ref{cor1} implies that $\nabla \tilde \Psi_m^*$ is locally $C^{1,\alpha}$ at $y$. Now thanks to Theorem \ref{Fii2}, $\nabla \tilde \Psi_m^*(y)-y$ is different from $0$ and gives the direction of the inward normal to $V_m$; hence, this normal is locally H\"{o}lder continuous with universal exponent $\alpha>0$.        
\end{proof}

Corollary \ref{cor2} confirms Figalli's prediction on the regularity of the free boundaries \cite[Remark 4.15]{Fi}. Next, we would like to understand the set $F$ that shows up in Corollary \ref{cor1}. Our aim in the next section is to prove that under suitable conditions on the domains $\Omega$ and $\Lambda$, the free boundaries $\overline{\partial U_m \cap \Omega}$ and $\overline{\partial V_m \cap \Lambda}$ are $C_{loc}^{1,\alpha}$ hypersurfaces away from the common region $\Omega \cap \Lambda$ and up to a ``small" singular set contained at the intersection of fixed with free boundary (inside the domains, the result follows from Corollary \ref{cor2}).

\section{Analysis of the singular set}
\label{Section4}

The goal of this section is to prove that away from the common region, the free boundaries are locally $C^{1,\alpha}$ outside of an $\mathcal{H}^{n-2}$ $\sigma$-finite set. To achieve this task, we need some regularity assumptions on the domains and initiate the analysis by developing a method which combines geometric measure theory and convex analysis. The following result is a general statement about projections of convex sets onto other convex sets and is a crucial tool in our study of the singular set.

\begin{prop}  \label{geom2b} Assume $\Omega \subset \mathbb{R}^n$ is a convex, bounded domain and $\Lambda \subset \mathbb{R}^n$ is a uniformly convex, bounded domain with $C^{1,1}$ boundary. Then $$\partial (P_\Lambda(\Omega)\cap \partial \Lambda) \setminus \partial (\partial(\Omega \cap \Lambda)\cap \partial \Lambda)$$ is locally $\mathcal{H}^{n-2}$ finite ($\partial P_\Lambda(\Omega)$ is discussed under Definition \ref{projj}).     
\end{prop}

\begin{proof}
If $\partial (P_\Lambda(\Omega)\cap \partial \Lambda)\setminus \partial (\partial(\Omega \cap \Lambda)\cap \partial \Lambda)=\emptyset,$ then there is nothing to prove (this is the case if $\Omega \subset \Lambda$). Let $y \in \partial (P_\Lambda(\Omega) \cap \partial \Lambda)\setminus \partial (\partial(\Omega \cap \Lambda)\cap \partial \Lambda)$. We may pick $\rho_y>0$ sufficiently small so that 
\begin{equation} \label{hear}
\overline{B_{\rho_y}(y)} \cap \big(\partial (P_\Lambda(\Omega) \cap \partial \Lambda)\setminus\partial (\partial(\Omega \cap \Lambda)\cap \partial \Lambda)\big) = \overline{B_{\rho_y}(y)} \cap \partial (P_\Lambda(\Omega) \cap \partial \Lambda)=:E.
\end{equation}
Our aim is to prove the existence of $\epsilon_y>0$ so that  
\begin{equation} \label{goal}
\mathcal{H}^{n-2}(\overline{B_{\epsilon_y}(y)}\cap E)<\infty.
\end{equation}
Since $y \in \partial \Lambda$, convexity of $\Lambda$ implies the existence of $r_y>0$ so that $B_{r_y}(y) \cap \partial \Lambda$ may be represented by the graph of a concave function $\phi_y$: $$\Lambda \cap B_{r_y} = \{z \in B_{r_y}(y): z_n < \phi_y(z_1, \ldots, z_{n-1})\}.$$ Without loss of generality, we may assume $\phi_y: B_{r_y}^{n-1}(\tilde y) \rightarrow \mathbb{R}$ with $y=(\tilde y, \phi(\tilde y)),$ and $N_\Lambda(y)=(0,1)$ so that $B_{r_y}^{n-1}(\tilde y)\subset \mathbb{T}_y \Lambda-(\tilde y, \phi_y(\tilde y)) \subset \mathbb{R}^{n-1}$ (recall Definition \ref{tang}). Now pick $\delta_y>0$ with $\delta_y\leq\frac{r_y}{2}$. Let $s_y:=\frac{\delta_y}{4}$ so that for all $\tilde z \in B_{\frac{\delta_y}{2}}^{n-1}(\tilde y)$ we have 
\begin{equation} \label{ano}
B_{s_y}^{n-1}(\tilde z) \subset B_{\frac{3\delta_y}{4}}^{n-1}(\tilde y).
\end{equation}
Fix $\tilde z\in B_{\frac{\delta_y}{2}}^{n-1}(\tilde y)$ and set $z:=(\tilde z, \phi_y(\tilde z)) \in \partial \Lambda$; there exists $r_z>0$ such that $\phi_z : B_{r_z}^{n-1}(\bar z) \rightarrow \mathbb{R}$ is a local parametrization of $\partial \Lambda$ at $z$ where $B_{r_z}^{n-1}(\bar z) \subset \mathbb{T}_z \Lambda-(\bar z, \phi_z(\bar z))$ (in this parametrization, $z=(\bar z, \phi_z(\bar z))$). Let $\Phi_y: B_{r_y}^{n-1}(\tilde y) \rightarrow \partial \Lambda$ be the map $\Phi_y(\tilde z)= (\tilde z, \phi(\tilde z))$ ($\Phi_z$ is similarly defined). Since $(\Phi_y^{-1} \circ \Phi_z)(\bar z)=\tilde z$, by continuity of $\Phi_y^{-1} \circ \Phi_z$, we may first pick $\eta=\eta(s_y)>0$ small enough so that $$\Phi_y^{-1}(\Phi_z(B_{\eta}^{n-1}(\bar z))) \subset B_{s_y}^{n-1}(\tilde z);$$ then by continuity of $\Phi_z^{-1} \circ \Phi_y$, there exists $\mu=\mu(\eta)>0$ so that $$B_{\mu}^{n-1}(\tilde z) \subset \Phi_y^{-1}(\Phi_z(B_{\eta}^{n-1}(\bar z))).$$ Thus, by (\ref{ano}) we obtain 
\begin{equation} \label{e3}
B_{\mu}^{n-1}(\tilde z) \subset \Phi_y^{-1}(\Phi_z(B_{\eta}^{n-1}(\bar z)))\subset B_{\frac{3\delta_y}{4}}^{n-1}(\tilde y).
\end{equation}

\vskip .1in 
\textit{Claim:} Let $w \in E$ (see \ref{hear}) and $\phi: B_s^{n-1}(\tilde w) \rightarrow \mathbb{R}$ be any concave parametrization of $\partial \Lambda$ at $w=(\tilde w, \phi(\tilde w))$ such that $N_\Lambda(w)=(0,1)$, and graph$(\phi)$ $\cap \partial (\partial(\Omega \cap \Lambda)\cap \partial \Lambda)=\emptyset$. Then, there exists an $(n-2)$-dimensional hyperplane $\pi(\tilde w)$ and a cone $C_\alpha(\pi(\tilde w)) \subset \mathbb{R}^{n-1}$ (see Definition \ref{cone}) with $\alpha=\alpha(\Lambda)$ so that $$\Phi^{-1}\big(\overline{P_\Lambda(\Omega)\cap \partial \Lambda} \big) \cap \big(\tilde w + C_\alpha(\pi(\tilde w))\big) = \emptyset,$$ where $\Phi:B_s^{n-1}(\tilde w) \rightarrow \partial \Lambda $ is the map $\Phi(\tilde x):=(\tilde x, \phi(\tilde x))$.  
\vskip .1in

\textit{Proof of Claim}: First, since $\partial \Lambda$ is uniformly convex, there is a constant $C_1>0$ such that for all $\tilde x, \tilde y$ 
\begin{equation} \label{co}
\langle \nabla \phi(\tilde y)-\nabla \phi(\tilde x), \tilde x- \tilde y \rangle \geq C_1|\tilde x- \tilde y|^2.
\end{equation} 
Moreover, let $$x:=w+t^*(w)N_\Lambda(w) \in \partial \Omega,$$ where $t^*(w):=\inf \{t\geq0: w+tN_\Lambda(w) \in \overline{\Omega} \}$. Note that since $$w \in E=\overline{B_{\rho_y}(y)} \cap \big(\partial (P_\Lambda(\Omega) \cap \partial \Lambda)\setminus \partial (\partial(\Omega \cap \Lambda) \cap \partial \Lambda)\big),$$ $t^*(w)>0$; hence, the half-line $\{L_t:=w+tN_\Lambda(w)\}_{t>0}$ touches $\Omega$ on the boundary at $x$ and lies on a tangent space of $\Omega$ at $x$ with normal $N_\Omega(x)$. This implies $\langle N_\Omega(x), N_\Lambda(w) \rangle =0$ and since $N_\Lambda(w)=(0,1)$ we have that $e_{n-1}:= N_\Omega(x) \in \mathbb{R}^{n-1}$ (since its $n$-th component is $0$). Next, let $\{e_1,\ldots, e_{n-1}\}$ be an orthonormal basis for $\mathbb{R}^{n-1}$ and fix $\tilde z \in B_s^{n-1}(\tilde w)$; thus, $\tilde z = \displaystyle \sum_{i=1}^{n-1} b_i e_i+\tilde w$ with $\bigg|\displaystyle \sum_{i=1}^{n-1} b_i e_i\bigg|\leq s$. Let $C_2>0$ be the uniform Lipschitz constant of $\partial \Lambda$ and define $\alpha:=\frac{C_2}{C_1}>0$. Set $\pi(\tilde w):=e_1^\perp=\mathbb{R}^{n-2}$, and define $$C_\alpha(\pi(\tilde w)):=\bigl\{(b_1,\ldots, b_{n-1})=(b_{n-1}^\perp, b_{n-1}) \in \mathbb{R}^{n-1}: \alpha |b_{n-1}^\perp| < b_{n-1} \bigr\}.$$ We will now show that $C_\alpha(\pi(\tilde w))$ is the desired cone. It suffices to show that if $\tilde z \in \tilde w+C_\alpha(\pi(w))$, then $\langle N_\Lambda((\tilde z, \phi(\tilde z)), e_{n-1}\rangle \geq 0$ since if this is true, then for $t\geq 0$,    
\begin{align*}
\langle z+tN_\Lambda(z)-x, e_{n-1}\rangle &= \langle z+tN_\Lambda(z)-(w+t^*(w) N_\Lambda(w)), N_\Omega(x) \rangle\\
&= \langle \tilde z- \tilde w, e_{n-1}\rangle+t\langle N_\Lambda(z), e_{n-1}\rangle- t^*(w) \langle N_\Lambda(w), N_\Omega(x)\rangle\\
&\geq b_{n-1} > \alpha |b_{n-1}^\perp|\geq 0,
\end{align*}            
so by convexity of $\Omega$, $z+tN_\Lambda(w) \notin \overline \Omega$, and this implies $\tilde z \notin \Phi^{-1}\bigl(\overline{P_\Lambda(\Omega)\cap \partial \Lambda}\bigr)$. Therefore, we will prove that if $\tilde z \in \tilde w+C_\alpha(\pi(\tilde w))$, then $\langle N_\Lambda((\tilde z, \phi(\tilde z)), e_{n-1}\rangle \geq 0$: since $N_\Lambda(z)=\frac{(-\nabla \phi(\tilde z), 1)}{\sqrt{1+|\nabla \phi(\tilde z)|^2}}$, it suffices to prove $\langle -\nabla \phi(\tilde z), e_{n-1}\rangle \geq 0$. Write $b_{n-1}^{\perp}:=\displaystyle \sum_{i=1}^{n-2} b_{i}e_i$ where $$\tilde z = b_{n-1}^{\perp}+ b_{n-1}e_{n-1}+ \tilde w.$$ Since $\nabla \phi(\tilde w)=0$, we may use (\ref{co}) and the fact that the Lipschitz constant of $\nabla \phi$ is $C_2$ to obtain
\begin{align*}
\langle \nabla \phi(\tilde z), e_{n-1}\rangle&=\langle \nabla \phi(\tilde z) - \nabla \phi(\tilde z - b_{n-1}^{\perp})+ \nabla \phi(\tilde z - b_{n-1}^{\perp})-\nabla \phi(\tilde w), e_{n-1} \rangle \\
& \leq C_2|b_{n-1}^{\perp}|+ \frac{1}{b_{n-1}}\langle \nabla \phi(b_{n-1}e_{n-1}+ \tilde w)-\nabla \phi(\tilde w), b_{n-1}e_{n-1} \rangle\\
& \leq C_2|b_{n-1}^{\perp}|- \frac{1}{b_{n-1}} C_1 |b_{n-1}e_{n-1}|^2\\
&\leq C_2\bigg(\frac{C_1}{C_2} b_{n-1}\bigg)- C_1b_{n-1} = 0\\
\end{align*}      

\textit{End of Claim.} \\

\hskip .1in 

\noindent Let $z\in \overline{B_{\frac{\delta_y}{2}}(y)} \cap E$. Without loss of generality, we may assume $\frac{\delta_y}{2}\leq \rho_y$. By the claim we obtain $$\Phi_z \Big(B_{\eta}^{n-1}(\bar z)\cap \big(\bar z + C_\alpha(\pi(\bar z))\big)\Big) \subset \partial \Lambda \setminus \overline{P_\Lambda(\Omega)}$$ so that by (\ref{e3}),
\begin{equation} \label{e1}
\Phi_y^{-1}\Big(\Phi_z\big(B_{\eta}^{n-1}(\bar z)\cap (\bar z+C_\alpha(\pi(\bar z)))\big)\Big) \subset B_{\frac{3}{4}\delta_y}(\tilde y) \cap \Phi_y^{-1}\big(\partial \Lambda\setminus \overline{P_\Lambda(\Omega)}\big).
\end{equation}
Now $\Phi_y^{-1}(\Phi_z(\bar z))=\tilde z$, and since $\Lambda$ is uniformly Lipschitz, $\Phi_y^{-1} \circ \Phi_z$ has a uniform Lipschitz constant; hence, there exists a cone $C_{\tilde \alpha}(\tilde \pi(\tilde z))$, where $\tilde \alpha$ depends only on the Lipschitz constant of $\Lambda$ and $\tilde \pi(\tilde z)$ is an $(n-2)$-dimensional hyperplane, for which 
\begin{equation} \label{e2}
(\tilde z + C_{\tilde \alpha}(\tilde \pi(\tilde z))) \cap \Phi_y^{-1}\big(\Phi_z(B_{\eta}^{n-1}(\bar z))\big) \subset \Phi_y^{-1}\Big(\Phi_z\big(B_{\eta}^{n-1}(\bar z) \cap (\bar z+C_\alpha(\pi(\tilde y)))\big)\Big).
\end{equation}
By (\ref{e3}), we obtain 
\begin{equation*}
(\tilde z + C_{\tilde \alpha}(\tilde \pi(\tilde z))) \cap B_{\mu}^{n-1}(\tilde z) \subset (\tilde z + C_{\tilde \alpha}(\tilde \pi(\tilde z))) \cap \Phi_y^{-1}(\Phi_z(B_{\eta}^{n-1}(\bar z))),
\end{equation*}
which combines with (\ref{e1}), and (\ref{e2}) to yield,  

$$(\tilde z + C_{\tilde \alpha}(\tilde \pi(\tilde z))) \cap B_{\mu}^{n-1}(\tilde z)\subset B_{\frac{3}{4}\delta_y}(\tilde y) \cap \Phi_y^{-1}\big(\partial \Lambda\setminus \overline{P_\Lambda(\Omega)}\big) \subset \Phi_y^{-1}\Big(\overline{B_{r_y}(y)} \cap \big(\partial \Lambda\setminus \overline{P_\Lambda(\Omega)}\big)\Big);$$
hence, 
$$(\tilde z + C_{\tilde \alpha}(\tilde \pi(\tilde z))) \cap B_{\mu}^{n-1}(\tilde z) \cap \Phi_y^{-1}\big(\overline{ B_{\frac{\delta_y}{2}}(y)} \cap E)= \emptyset.$$ 

\noindent Now applying Corollary \ref{corgeom}, we obtain that $\Phi_y^{-1}\big(\overline{ B_{\frac{\delta_y}{2}}(y)} \cap E)$ is finitely $(n-2)$-rectifiable. Since $\phi_y$ is bi-Lipschitz with uniform Lipschitz constant, we have 
\begin{equation} \label{e11}
\mathcal{H}^{n-2}\big(\overline{B_{\frac{\delta_y}{2}}(y)} \cap E \big) < \infty.
\end{equation}
This yields (\ref{goal}) and finishes the proof.

\end{proof}

\begin{cor} \label{inve}
Assume $\Omega \subset \mathbb{R}^n$ is a convex, bounded domain and $\Lambda \subset \mathbb{R}^n$ is a uniformly convex, bounded domain with $C^{1,1}$ boundary. If $\overline{\Omega} \cap \overline{\Lambda}=\emptyset$, then $$\mathcal{H}^{n-2}\big(\partial(P_\Lambda(\Omega)\cap \partial \Lambda)\big)<\infty.$$    
\end{cor}

\begin{proof}
Simply note that $$\partial (P_\Lambda(\Omega)\cap \partial \Lambda) \setminus \partial (\partial(\Omega \cap \Lambda)\cap \partial \Lambda)=\partial (P_\Lambda(\Omega)\cap \partial \Lambda),$$ and as the latter set is compact, the result follows from Proposition \ref{geom2b}.

\end{proof}

Note that Proposition \ref{geom2b} is a purely geometric result. We will now connect this geometry with the optimal partial transport problem.

\begin{lem} \label{yay} Assume $\Omega \subset \mathbb{R}^n$ is a strictly convex, bounded domain and $\Lambda \subset \mathbb{R}^n$ is a uniformly convex, bounded domain with $C^{1,1}$ boundary. Let $\tilde\Psi_m$ be the $C^1(\mathbb{R}^n)$ extension of $\Psi_m$ to $\mathbb{R}^n$ given by Theorem \ref{Fii1}. Then 

\begin{equation} \label{dfg}
\nabla \tilde \Psi_m(A_2)\subset \partial (P_\Lambda(\Omega) \cap \partial \Lambda)\setminus \partial (\partial(\Omega \cap \Lambda)\cap \partial \Lambda),
\end{equation}
 with $A_2$ as in Definition \ref{d1}. 
\end{lem}

\begin{proof}
Let $y=\nabla \tilde \Psi_m(x) \in \nabla \tilde \Psi_m(A_2)$. Since $A_2 \subset S_1$, we have $\nabla \tilde \Psi_m(x)\neq x$ and $$y \notin \partial (\partial(\Omega \cap \Lambda)\cap \partial \Lambda).$$ Moreover, let $L_t:=\nabla \tilde \Psi_m(x)+ \frac{x-\nabla \tilde \Psi_m(x)}{|x-\nabla \tilde \Psi_m(x)|}t$ and note that the half-line $\{L_t\}_{t\geq0}$ is tangent to the active region. Since $x \in \partial \Omega \setminus \partial U_m$, it follows that $L_t$ is tangent to $\Omega$ at $x$; hence, it is on a tangent space to $\Omega$ at $x$. Next, let $z=P_\Lambda(x) \in \partial \Lambda$ (recall that $P_\Lambda$ is the orthogonal projection operator). Then by the properties of the projection operator (and the convexity of $\Lambda$), we know that $x-z$ is parallel to $N_\Lambda(z)$. Since $x \in S_1$, it follows that $\nabla \tilde \Psi_m(x) \in \partial_{nt} \Lambda$; in particular, $x-\nabla \tilde \Psi_m(x)$ is parallel to $N_\Lambda(\nabla \tilde \Psi_m(x))$. Thus, by uniqueness of the projection, it readily follows that $z=\nabla \tilde \Psi_m(x)=y$. Combining $\{L_t\}_{t\geq0} \subset \mathbb{T}_x \Omega$ and $y=P_\Lambda(x)$ yields $y \in \partial (P_\Lambda(\Omega)\cap \partial \Omega)$.              
\end{proof}  

Next, we turn our attention to the set $A_1$. Recall $S_1=A_1 \cup A_2$, and as evidenced by Lemma \ref{yay}, the set $A_2$ has a rich geometric structure. Analogously, the next proposition provides insight into the geometry of $A_1$ (via Corollary \ref{sig}).   

\begin{prop} (Nontransverse intersection points never map to nontransverse intersection points)  \label{propo}
Suppose that $\Omega \subset \mathbb{R}^n$ and $\Lambda \subset \mathbb{R}^n$ are bounded, strictly convex domains, and let $\tilde\Psi_m$ be the $C^1(\mathbb{R}^n)$ extension of $\Psi_m$ to $\mathbb{R}^n$ given by Theorem \ref{Fii1}. Then $$\nabla \tilde \Psi_m(\partial_{nt} \Omega) \cap \partial_{nt} \Lambda = \emptyset,$$ where $\partial_{nt} \Lambda$ (and by duality $\partial_{nt} \Omega$) is defined in (\ref{nt}). 
\end{prop}

\begin{proof}
Let $$\nabla \tilde \Psi_m(x)=: y\in \nabla \tilde \Psi_m(\partial_{nt} \Omega) \cap \partial_{nt} \Lambda$$ and suppose $\Omega \cap \Lambda \not = \emptyset$. If $x=y$, then by strict convexity, $$\langle N_\Lambda(x), z-x\rangle<0$$ for all $z \in \overline \Lambda$. However, we also have $N_\Omega(x)=-N_\Lambda(x)$ (since $x=y\in \nabla \tilde \Psi_m(\partial_{nt} \Omega) \cap \partial_{nt} \Lambda$) so that $$\langle N_\Omega(x), z-x\rangle>0$$ for all $z \in \overline \Omega$. Now, pick $z\in \Omega \cap \Lambda$; then from the convexity of $\Omega$ we have $$\langle N_\Omega(x), z-x\rangle\leq0,$$ a contradiction. Therefore, we may assume without loss of generality that  $x \not = y$. By definition of $\partial_{nt} \Omega$ and $\partial_{nt} \Lambda$, $y-x$ is parallel to a normal of $\Omega$ at $x$ and $x-y$ is parallel to a normal to $\Lambda$ at $y$. Using the strict convexity of $\Lambda$ and convexity of $\Omega$, this means that for $z \in \Lambda \cap \Omega$, $$\langle x-y, z-y\rangle < 0,$$ and $$\langle y-x, z-x\rangle \leq 0.$$ Thus, $$0<|x-y|^2=\langle x-y, x-y\rangle = \langle x-y, x-z\rangle+ \langle x-y, z-y\rangle <0,$$ a contradiction. Therefore, we have reduced the problem to the case when $\Omega \cap \Lambda = \emptyset$. Suppose $\nabla \tilde \Psi_m(x_0) \in \nabla \tilde \Psi_m(\partial_{nt} \Omega) \cap \partial_{nt} \Lambda$ and let $x_1 \in \partial U_m \cap \Omega$. By strict convexity of $\Lambda$, note that $d:=dist\big(\nabla \tilde \Psi_m(x_1),\mathbb{T}_{\nabla \tilde \Psi_m(x_0)}\Lambda\big)>0$ and 
\begin{equation} \label{zzz}
|\nabla \tilde \Psi_m(x_0)-x_0|+d \leq |\nabla \tilde \Psi_m(x_1)-x_1|.
\end{equation}

\noindent By continuity of $\nabla \tilde \Psi_m$, for $\epsilon>0$, there exists $\delta=\delta(\epsilon)>0$ such that $$\nabla \tilde \Psi_m(B_\delta(x_1) \cap U_m) \subset B_\epsilon(\nabla \tilde \Psi_m(x_1)) \cap V_m.$$ Now let $A_\delta:=B_\delta(x_1) \cap U_m$, and for $\eta>0$, set $A_\eta:=B_\eta(x_0) \cap U_m^c \cap \Omega$. Pick $\epsilon>0$ small enough so that $A_\eta \cap A_\delta =\emptyset$, see Figure \ref{xenon2}. Then, by possibly reducing $\epsilon>0$ (thereby also reducing $\delta$), we may pick $\eta=\eta(\epsilon)>0$ small so that 
\begin{equation} \label{eq}
\int_{A_\eta} f(x)dx= \int_{A_\delta} f_m(x)dx.
\end{equation}
\begin{figure}[h!]
\centering 
\includegraphics[scale= .5]{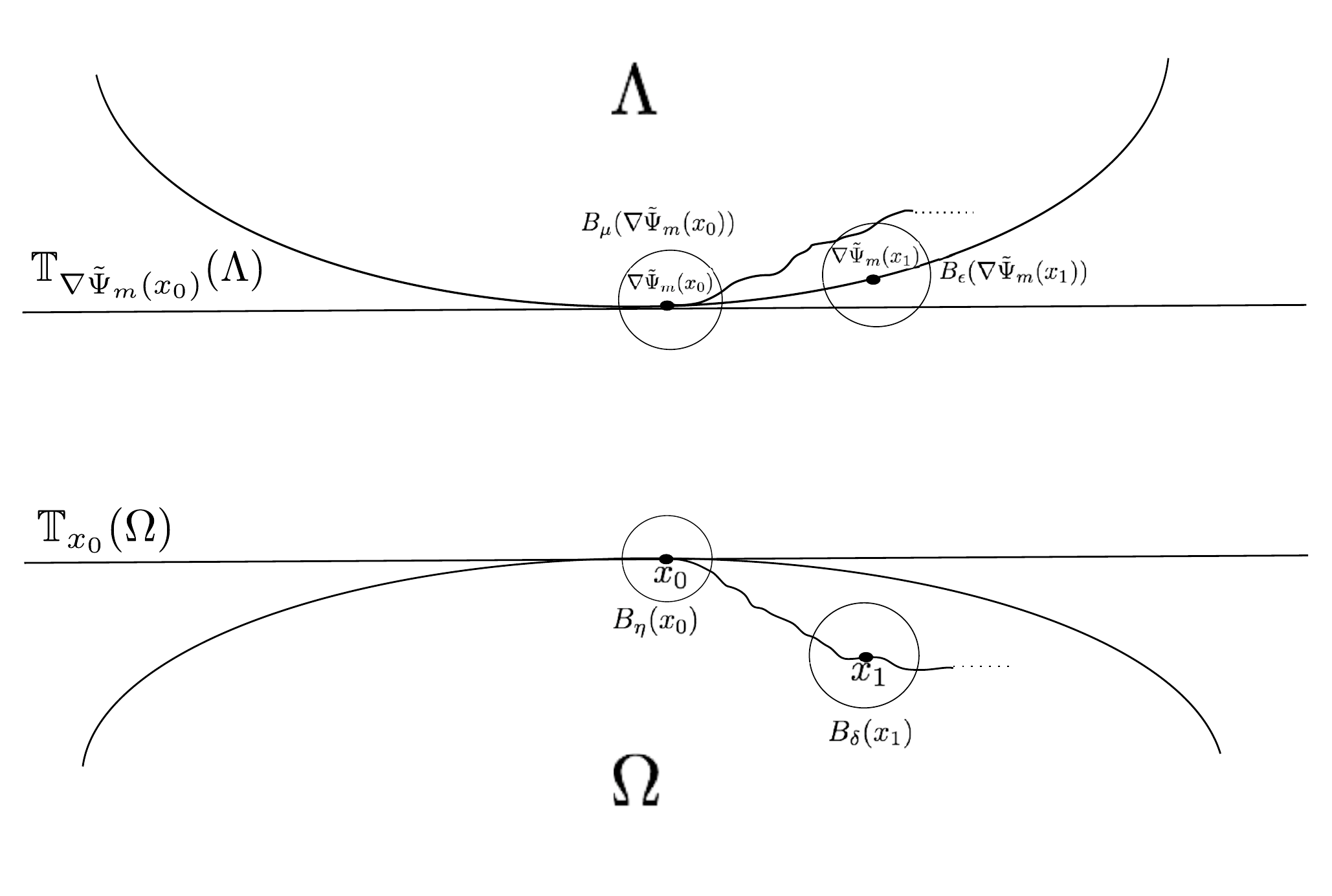}
\caption{Constructing a cheaper transference plan.}
\label{xenon2}
\end{figure} 
\noindent Next, let $\mu=\mu(\epsilon)>0$ be small enough so that $$\int_{A_\eta} f(x)dx = \int_{B_\mu(\nabla \tilde \Psi_m(x_0))\cap V_m^c} g(x) dx,$$ and let $$T_\epsilon: A_\eta \rightarrow B_\mu(\nabla \tilde \Psi_m(x_0))\cap V_m^c \cap \Lambda$$ be the optimal transport map between $f \chi_{A_\eta}$ and 
$g \chi_{D_\mu}$, where $$D_\mu:=B_\mu(\nabla \tilde \Psi_m(x_0))\cap V_m^c \cap \Lambda.$$

\noindent Define
\begin{displaymath}
   \tilde T(x):= \left\{
     \begin{array}{lr}
      T_\epsilon(x), & \hskip .05in x \in A_\eta \\ 
       x, \hskip .05in & \hskip 0.05in x \in A_\delta\\
      \nabla \tilde \Psi_m(x),  & \hskip .05in  x\in U_m \setminus B_\delta(x_1),
     \end{array}
   \right.
\end{displaymath}       

\begin{displaymath}
   \tilde f(x):= \left\{
     \begin{array}{lr}
      f(x), & \hskip .05in x \in A_\eta \\ 
       f_m(x), \hskip .05in & \hskip 0.05in x \in U_m \setminus B_\delta(x_1)\\
      0,  & \hskip .05in  otherwise.
     \end{array}
   \right.
\end{displaymath}

\noindent Set $\tilde \gamma:=(Id \times \tilde T)_{\#}\tilde f$; it is easy to check that $\tilde \gamma$ is admissible. Now let $z \in A_\eta$ and $w \in A_\delta$ and select $\epsilon$ small enough so that $$\eta(\epsilon)+\mu(\epsilon) + \delta(\epsilon)+\epsilon<\frac{d}{2}.$$ Then, by (\ref{zzz}) and the triangle inequality we obtain
\begin{align*}
|z-T_\epsilon(z)|&\leq |z-x_0|+|x_0-\nabla \tilde \Psi_m(x_0)|+|\nabla \tilde \Psi_m(x_0)-T_\epsilon(z)| \nonumber \\
&\leq \eta(\epsilon)+\mu(\epsilon)+ |x_1-\nabla \tilde \Psi_m(x_1)| - d \nonumber \\
&\leq \eta(\epsilon)+\mu(\epsilon)+ |x_1-w|+|w-\nabla \tilde \Psi_m(w)|+|\nabla \tilde \Psi_m(w)-\nabla \tilde \Psi_m(x_1)| - d \nonumber\\
&\leq \eta(\epsilon)+\mu(\epsilon) + \delta(\epsilon)+\epsilon - d + |w-\nabla \tilde \Psi_m(w)| \nonumber\\
&\leq |w-\nabla \tilde \Psi_m(w)|-\frac{d}{2}. \nonumber \\
\end{align*}
This shows that the cost of $\tilde T$ inside $A_\eta$ is strictly less than the one of $\nabla \tilde \Psi_m$ inside $A_\delta$, and since these maps coincide elsewhere, this contradicts the minimality of $\nabla \tilde \Psi_m$.

\end{proof}

Proposition \ref{propo} enables us to apply a weak form of the implicit function theorem to prove that $A_1$ is $\mathcal{H}^{n-2}$ $\sigma$-finite; moreover, this information combines with the geometry established in the proof of Lemma \ref{yay} to estimate the size of $\nabla \tilde \Psi_m(A_1)$. This is the content  of the following two corollaries.

\begin{cor} \label{sig}
Assume that $\Omega \subset \mathbb{R}^n$ and $\Lambda \subset \mathbb{R}^n$ are bounded, strictly convex domains. Then the relatively closed set $A_1$ (see Remark \ref{bbgun}) is $\mathcal{H}^{n-2}$ $\sigma$-finite. Moreover, if $\Omega$ has a $C^1$ boundary, then there exists $\{x_k\}_{k=1}^\infty \subset A_1$ and $R_k>0$ such that 
\begin{equation} \label{covr}
A_1 \subset \bigcup_{k=1}^\infty B_{R_k}(x_k),
\end{equation}
with $\mathcal{H}^{n-2}(A_1 \cap B_{R_k}(x_k))<\infty$. If in addition $\overline \Omega \cap \overline \Lambda = \emptyset$, then $\mathcal{H}^{n-2}(A_1)<\infty$.      

\end{cor}

\begin{proof} Let $D_\Omega$ denote the set of differentiability points of $\partial \Omega$ and set $$A_1^1:= A_1 \cap \partial_{nt} \Omega, \hskip .1in A_1^2:=(A_1 \setminus \partial_{nt} \Omega) \cap D_\Omega.$$ If $x \in A_1^1$, then $\nabla \tilde \Psi_m(x) \in \partial_{nt} \Lambda$. Therefore, $\nabla \tilde \Psi_m(A_1^1) \subset \nabla \tilde \Psi_m(\partial_{nt} \Omega) \cap \partial_{nt} \Lambda$, so by Proposition \ref{propo}, $A_1^1 = \emptyset$. Next, since $\Omega \subset \mathbb{R}^n$ is convex,  it is well-known that the set of non-differentiability points is $\mathcal{H}^{n-2}$ $\sigma$-finite (see for instance \cite{Alb}); thus, $(A_1 \setminus \partial_{nt} \Omega) \setminus D_\Omega$ is $\mathcal{H}^{n-2}$ $\sigma$-finite. Now let $x \in A_1^2$ so that by Remark \ref{r0}, $x \in (\partial U_m \cap \partial \Omega) \setminus \partial_{nt} \Omega$. Therefore, at $x$, the free boundary $\partial U_m$ touches the fixed boundary transversally so that $N_{U_m}(x) \neq N_{\Omega}(x)$ and since $x$ is a differentiability point of $\Omega$, we may apply the weak implicit function theorem (see e.g. \cite[Corollary 10.52]{Vi}) to obtain $R(x)>0$ such that $\partial U_m \cap \partial \Omega \cap B_{R(x)}(x)$ is contained in an $(n-2)$-dimensional Lipschitz graph. In particular, $$\mathcal{H}^{n-2}(\partial U_m \cap \partial \Omega \cap B_{R(x)}(x)) < \infty.$$ Now 
\begin{equation} \label{juo}
A_1^2 \subset \bigcup_{x\in A_1^2} B_{R(x)}(x);
\end{equation}
thus, there exists $\{x_k\}_{k=1}^{\infty} \subset A_1^2$ such that $$A_1^2 = \bigcup_{k=1}^{\infty} \bigl(A_1^2 \cap B_{R_k}(x_k)\bigr).$$ Since $(A_1 \setminus \partial_{nt} \Omega) \setminus D_\Omega$ is $\mathcal{H}^{n-2}$ $\sigma$-finite, it readily follows that
\begin{equation} \label{gah}
A_1=\bigcup_{k=0}^\infty E_k,
\end{equation}
with $$\mathcal{H}^{n-2}(E_k) < \infty.$$ This proves the first part of the corollary. If $\Omega$ has a $C^1$ boundary, then $A_1=A_1^2$ so (\ref{covr}) follows from (\ref{juo}). Furthermore, if $\overline \Omega \cap \overline \Lambda = \emptyset$, then $S_2=\emptyset$ and $X_s=S_1$ is compact by Lemma \ref{s2}; this implies that $A_1=A_1^2$ is compact; thus, using (\ref{juo}), we may extract a finite subcover to conclude the proof.      

\end{proof}

\begin{cor} \label{sigc}
Assume $\Omega \subset \mathbb{R}^n$ and $\Lambda \subset \mathbb{R}^n$ are bounded, strictly convex domains, and let $\tilde\Psi_m$ be the $C^1(\mathbb{R}^n)$ extension of $\Psi_m$ to $\mathbb{R}^n$ given by Theorem \ref{Fii1}. Then the relatively closed set $\nabla \tilde \Psi_m(A_1)$ (see Remark \ref{bbgun}) is $\mathcal{H}^{n-2}$ $\sigma$-finite. Moreover, if $\Omega$ has a $C^1$ boundary, then there exist open sets $F_k \subset \partial(\Lambda \cap V_m)$ (in the subspace topology) such that 
\begin{equation} \label{covr2}
\nabla \tilde \Psi_m(A_1) \subset \bigcup_{k=1}^\infty F_k,
\end{equation}   
with $\mathcal{H}^{n-2}(\nabla \tilde \Psi_m(A_1)\cap F_k)<\infty$. If in addition $\overline \Omega \cap \overline \Lambda = \emptyset$, then $$\mathcal{H}^{n-2}(\nabla \tilde \Psi_m(A_1))<\infty.$$     
\end{cor}

\begin{proof}
If $y \in \nabla \tilde \Psi_m(A_1)$, then $y=\nabla \tilde \Psi_m(x)$ with $x \in A_1 \subset S_1$; in particular, $y \in \partial_{nt} \Lambda$ so that $x-\nabla \tilde \Psi_m(x)$ is parallel to a normal of $\Lambda$ at $\nabla \tilde \Psi_m(x) \in \partial \Lambda$. Hence, $\nabla \tilde \Psi_m(x)=P_\Lambda(x)$, so that 
\begin{equation} \label{jkl}
\nabla \tilde \Psi_m(A_1)=P_\Lambda(A_1).
\end{equation}
Now from (\ref{gah}) in the proof of Corollary \ref{sig}, $A_1=\bigcup_{k=0}^\infty E_k$, with $\mathcal{H}^{n-2}(E_k) < \infty$, so $$P_\Lambda(A_1)=\bigcup_{k=0}^\infty P_\Lambda(E_k),$$ and since $P_\Lambda$ is Lipschitz, $\mathcal{H}^{n-2}(P_\Lambda(E_k))\leq \mathcal{H}^{n-2}(E_k)<\infty$, and this proves that $\nabla \tilde \Psi_m(A_1)$ is $\mathcal{H}^{n-2}$ $\sigma$-finite. If $\Omega$ has a $C^1$ boundary, then we may use (\ref{covr}) to define $F_k:=\nabla \tilde \Psi_m(B_{R_k}(x_k) \cap \partial (\Omega \cap U_m))$; note that since $\nabla \tilde \Psi_m$ is a homeomorphism between the active regions, each $F_k$ is open in $\partial(\Lambda \cap V_m)$. Moreover, thanks to (\ref{jkl}), $$\mathcal{H}^{n-2}(\nabla \tilde \Psi_m(A_1) \cap F_k)=\mathcal{H}^{n-2}(P_\Lambda(A_1)\cap F_k) \leq \mathcal{H}^{n-2}(A_1 \cap B_{R_k}(x_k))<\infty,$$  
and we obtain (\ref{covr2}). If in addition $\overline \Omega \cap \overline \Lambda = \emptyset$, then Corollary \ref{sig} implies $\mathcal{H}^{n-2}(P_\Lambda(A_1))\leq \mathcal{H}^{n-2}(A_1) <\infty$.   
\end{proof}

Since $S_1=A_1 \cup A_2$, we are now in a position to prove that the set $\nabla \tilde \Psi_m(S_1)$ is $\mathcal{H}^{n-2}$ $\sigma$-finite.

\begin{prop} \label{sigc2}
Assume $\Omega \subset \mathbb{R}^n$ is a strictly convex, bounded domain with $C^1$ boundary and $\Lambda \subset \mathbb{R}^n$ is a uniformly convex, bounded domain with $C^{1,1}$ boundary. Let $\tilde\Psi_m$ be the $C^1(\mathbb{R}^n)$ extension of $\Psi_m$ to $\mathbb{R}^n$ given by Theorem \ref{Fii1}. Then the relatively closed set $\nabla \tilde \Psi_m(S_1)$ (see Corollary \ref{impa}) is $\mathcal{H}^{n-2}$ $\sigma$-finite. Moreover, if $\overline \Omega \cap \overline \Lambda = \emptyset$, then $$\mathcal{H}^{n-2}(\nabla \tilde \Psi_m(S_1))<\infty.$$   

\end{prop}

\begin{proof}
Recall $\nabla \tilde \Psi_m(A_2) \subset \partial(P_\Lambda(\Omega) \cap \partial \Lambda) \setminus \partial (\partial(\Omega \cap \Lambda)\cap \partial \Lambda)$ (see (\ref{dfg})). Thus, for each $y \in \nabla \tilde \Psi_m(A_2)$, Proposition \ref{geom2b} implies the existence of $R_y>0$ such that $$\mathcal{H}^{n-2}(B_{R_y}(y)\cap \nabla \tilde \Psi_m(A_2))<\infty.$$ Now pick $\{y_k\}_{k\in \mathbb{N}} \subset \nabla \tilde \Psi_m(A_2)$ with this property so that $$\nabla \tilde \Psi_m(A_2)\subset \bigcup_{k=1}^\infty B_{R_{y_k}}(y_k).$$ Set $\tilde F_k^2:= B_{R_{y_k}}(y_k) \cap \nabla \tilde \Psi_m(A_2)\cap  \partial(\Lambda \cap V_m)$ and note that since $\nabla \tilde \Psi_m(A_2) \subset \partial(\Lambda \cap V_m)$, $$\nabla \tilde \Psi_m(A_2) = \bigcup_{k=1}^\infty \tilde F_k^2,$$ with $\mathcal{H}^{n-2}(\tilde F_k^2)<\infty$. Recalling $S_1=A_1\cup A_2$ and setting $\tilde F_k^1:=\nabla \tilde \Psi_m(A_1)\cap F_k$, an application of Corollary \ref{sigc} yields the $\mathcal{H}^{n-2}$ $\sigma$-finiteness with $\tilde F_k:=\tilde F_k^1 \cup \tilde F_k^2$.
Finally, if $\overline \Omega \cap \overline \Lambda = \emptyset$, then thanks to Corollary \ref{sigc}, $$\mathcal{H}^{n-2}(\nabla \tilde \Psi_m(A_1))<\infty,$$  and $$\mathcal{H}^{n-2}(\nabla \tilde \Psi_m(A_2))<\infty$$ by Corollary \ref{inve} and (\ref{dfg}).  

\end{proof}

Before proving the main result of this section (i.e. Theorem \ref{thhm}), we need one more statement about the size of the set consisting of points at the intersection of the target free boundary with fixed boundary that are the image of corresponding points at the intersection of the source free boundary with fixed boundary under the partial transport map.  

\begin{prop} \label{propo23} Assume $\Omega \subset \mathbb{R}^n$ and $\Lambda \subset \mathbb{R}^n$ are bounded, strictly convex domains, and let $$G:=\nabla \tilde \Psi_m(\partial U_m \cap \partial \Omega)\cap \partial V_m \cap \partial \Lambda,$$ where $\tilde\Psi_m$ is the $C^1(\mathbb{R}^n)$ extension of $\Psi_m$ to $\mathbb{R}^n$ given by Theorem \ref{Fii1}. Then $G$ admits a decomposition $G=G_1 \cup G_2$, where $G_1$ is relatively closed in $\partial \Lambda \setminus \partial \Omega$ and $G_2$ is compact with $\mathcal{H}^{n-2}$ finite measure. Moreover,   
\begin{equation} \label{covr4}
G_1 \subset \bigcup_{k=1}^\infty B_{R_k}(y_k),
\end{equation}
with $\mathcal{H}^{n-2}(G_1 \cap B_{R_i}(y_i)) < \infty$. If in addition, $\overline \Omega \cap \overline \Lambda =\emptyset$, then we also have that $G_1$ is compact with $\mathcal{H}^{n-2}(G_1)<\infty.$ 
\end{prop}

\begin{proof}

Consider the decomposition $$G=G_1\cup G_2,$$ with $$G_1:= (\nabla \tilde \Psi_m(\partial U_m \cap \partial \Omega)\setminus \partial_{nt} \Lambda)\cap (\partial V_m \cap \partial \Lambda),$$ and $$G_2:=\nabla \tilde \Psi_m(\partial U_m \cap \partial \Omega)\cap \partial_{nt} \Lambda.$$ Note that $G_2$ is compact; furthermore, split $G_2=G_2^1 \cup G_2^2$, with $$G_2^1:=\nabla \tilde \Psi_m(\partial_{nt} \Omega)\cap \partial_{nt} \Lambda,$$ $$G_2^2:=\nabla \tilde \Psi_m(\partial U_m \cap \partial \Omega \setminus \partial_{nt} \Omega)\cap \partial_{nt} \Lambda.$$ Using Proposition \ref{propo}, $G_2^1=\emptyset$. Next, observe that $K:= \partial U_m \cap \partial \Omega \cap (\nabla \tilde \Psi_m)^{-1}(\partial_{nt} \Lambda)$ is compact and by Proposition \ref{propo}, $$K=\partial U_m \cap \partial \Omega \cap (\nabla \tilde \Psi_m)^{-1}(\partial_{nt} \Lambda) \setminus \partial_{nt} \Omega.$$ Applying the weak implicit function theorem, we have that for all $x \in (\partial U_m \cap \partial \Omega) \setminus \partial_{nt} \Omega$, there exists $R(x)>0$ such that $$\mathcal{H}^{n-2}(B_{R(x)}(x)\cap \partial U_m \cap \partial \Omega)<\infty.$$  Therefore, by compactness, there exists $M \in \mathbb{N}$ and $\{x_i\}_{i=1}^M \subset K$ such that $$K \subset \bigcup_{i=1}^M B_{R(x_i)}(x_i).$$ Furthermore, recall that for $x \in (\nabla \tilde \Psi_m)^{-1}(\partial_{nt} \Lambda),$ $\nabla \tilde \Psi_m(x)=P_\Lambda(x)$ (see e.g. the proof of Lemma \ref{yay}). Hence, 
\begin{align}
\mathcal{H}^{n-2}(G_2)&=\mathcal{H}^{n-2}(\nabla \tilde \Psi_m(K))=\mathcal{H}^{n-2}(P_\Lambda(K)) \nonumber \\
&\leq \mathcal{H}^{n-2}(K) \leq \sum_{i=1}^M \mathcal{H}^{n-2}(B_{r(x_i)}(x_i)\cap K)<\infty. \label{newr}
\end{align} 
Now we show that $G_1$ is $\mathcal{H}^{n-1}$ $\sigma$-finite. Indeed, by applying the weak implicit function theorem once more, it follows that for all $y \in G_1 \subset (\partial V_m \cap \partial \Lambda)\setminus \partial_{nt} \Lambda$, there exists $R(y)>0$ such that $$\mathcal{H}^{n-2}(G_1 \cap B_{R(y)}(y))<\infty;$$ thus, we can find $\{y_k\}_{k=1}^\infty \subset G_1$ for which 
\begin{equation*} 
G_1 \subset \bigcup_{i=1}^\infty B_{R(y_i)}(y_i),
\end{equation*}
with $\mathcal{H}^{n-2}(G_1 \cap B_{R(y_i)}(y_i))<\infty$; this proves (\ref{covr4}). Next, assume further $\overline \Omega \cap \overline \Lambda = \emptyset.$ We claim that $G_1$ is compact. Indeed, let $y_n \in G_1$ with $y_n \rightarrow y \in \partial V_m \cap \partial \Lambda$. Set $x_n:=\nabla \tilde \Psi_m^*(y_n)$ and note that by continuity, $x_n \rightarrow x=\nabla \tilde \Psi_m^*(y)$. Since $\partial U_m \cap \partial \Omega$ is closed, it follows that $x \in \partial U_m \cap \partial \Omega$, so in particular $x\neq y$ (since $\overline \Omega \cap \overline \Lambda = \emptyset$). But we already know that $y \in  \nabla \tilde \Psi_m(\partial U_m \cap \partial \Omega) \cap (\partial V_m \cap \partial \Lambda)$; thus, it remains to show $y \not \in\partial_{nt} \Lambda$. If $y \in \partial_{nt} \Lambda$, strict convexity of $\Lambda$ implies that for all $z \in \overline{\Lambda}$, $$\langle x-y, z-y\rangle<0.$$ Since $y_n \not \in \partial_{nt} \Lambda$, for each $n \in \mathbb{N}$, there exists $z_n \in \Lambda$ for which $$\langle x_n-y_n, z_n-y_n \rangle \geq 0.$$ Now since $\overline \Lambda$ is compact, up to a subsequence, $z_n \rightarrow z \in \overline \Lambda$. Taking limits, it follows that $$\langle x-y, z-y\rangle\geq 0,$$ a contradiction; hence, $y \not \in \partial_{nt} \Lambda$, and so $G_1$ is compact (a similar argument shows that $G_1$ is relatively closed in $\partial \Lambda \setminus \partial \Omega$); thus, we may replace the infinite union in (\ref{covr4}) with a finite one to deduce $\mathcal{H}^{n-2}(G_1)< \infty$. 
\end{proof}

Now we have all the ingredients to prove that the free boundaries are local $C^{1,\alpha}$ hypersurfaces up to an explicit $\mathcal{H}^{n-2}$ $\sigma$-finite set, which is relatively closed in $(\partial \Omega \cup \partial \Lambda) \setminus \partial(\Omega \cap \Lambda)$. 

\begin{thm} \label{thhm} Assume $\Omega \subset \mathbb{R}^n$ and $\Lambda \subset \mathbb{R}^n$ are bounded, uniformly convex domains with $C^{1,1}$ boundaries, and let $\tilde\Psi_m$ be the $C^1(\mathbb{R}^n)$ extension of $\Psi_m$ to $\mathbb{R}^n$ given by Theorem \ref{Fii1}. Then away from $\partial(\Omega \cap \Lambda)$, the free boundary $\overline{\partial V_m \cap \Lambda}$ is a $C_{loc}^{1,\alpha}$ hypersurface up to the $\mathcal{H}^{n-2}$ $\sigma$-finite set: 
$$S:= \bigl((\nabla \tilde \Psi_m(\partial_{nc}U_m) \cup \nabla \tilde \Psi_m(S_1)) \cap \partial V_m \cap \partial \Lambda \bigr) \setminus \partial \Omega.$$ Moreover, $S$ is relatively closed in $\partial \Lambda \setminus \partial\Omega$, and 
if $\overline \Omega \cap \overline \Lambda = \emptyset$, then the free boundary $\overline{\partial V_m \cap \Lambda}$ is a $C_{loc}^{1,\alpha}$ hypersurface away from the compact, $\mathcal{H}^{n-2}$ finite set: $$S_d:= (\nabla \tilde \Psi_m(\partial_{nc}U_m) \cup \nabla \tilde \Psi_m(S_1)) \cap \partial V_m \cap \partial \Lambda \large.$$ By duality and symmetry, an analogous statement holds for $\overline{\partial U_m \cap \Omega}$.     
\end{thm}

\begin{proof}
Let $$F:=\nabla \tilde \Psi_m(\partial_{nc}U_m) \cup \nabla \tilde \Psi_m(S_1 \cup S_2)$$ as in Corollary \ref{cor1} so that $\overline{\partial V_m \cap \Lambda}$ is $C_{loc}^{1,\alpha}$ away from $F$; now recall that $\nabla \tilde \Psi_m(S_2)=S_2 \subset \partial(\Omega \cap \Lambda)$ (Remark \ref{r1}). Hence, the singular set for $\overline{\partial V_m \cap \Lambda}$ away from $\partial(\Omega \cap \Lambda)$ is $S$. Now let $$S_{tr}:=\big((\nabla \tilde \Psi_m(\partial_{nc}U_m) \setminus \nabla \tilde \Psi_m(S_1))\cap \partial V_m \cap \partial \Lambda \big) \setminus \partial (\Omega \cap \Lambda),$$ ($tr$ stands for ``transverse") and note $$S=S_{tr} \cup (\nabla \tilde \Psi_m(S_1)\setminus \partial (\Omega \cap \Lambda)),$$ (indeed, recall that the free boundary never enters the common region: Remark \ref{cmr}). For $y \in S_{tr}$, set $x:=\nabla \tilde \Psi_m^*(y)$; since $\Omega$ is convex and $x\in \partial_{nc}U_m$, it follows that $x \notin \partial \Omega \setminus \partial U_m$. Moreover, since free boundary never maps to free boundary (by Proposition \ref{fbd}), we also have $x \notin \partial U_m \cap \Omega$, which implies $x \in \partial U_m \cap \partial \Omega$. In particular, $$S_{tr} \subset G,$$ where $G$ is the set from Proposition \ref{propo23}. Therefore, 
\begin{equation} \label{erd}
S\subset (G \cup \nabla \tilde \Psi_m(S_1)) \setminus \partial(\Omega \cap \Lambda),
\end{equation}
and so combining Proposition \ref{sigc2} with Proposition \ref{propo23} yields that $S$ is $\mathcal{H}^{n-2}$ $\sigma$-finite. 
Next, since $\partial_{nc} U_m$ is a closed set, Corollary \ref{impa} implies that $S$ is relatively closed in $\partial \Lambda \setminus \partial\Omega$.  
To prove the last part of the theorem, assume $\overline \Omega \cap \overline \Lambda = \emptyset$. Then, $S=S_d$ is closed, hence, compact; moreover, invoking Propositions \ref{sigc2} \& \ref{propo23}  again, we have $\mathcal{H}^{n-2}(\nabla \tilde \Psi_m(S_1)) < \infty$ and $\mathcal{H}^{n-2}(G) < \infty$, so $\mathcal{H}^{n-2}(S)<\infty$ by (\ref{erd}).    

\end{proof}

\begin{rem} \label{imprem} In the non-disjoint case, the $\mathcal{H}^{n-2}$ $\sigma$-finite singular set $S$ from Theorem \ref{thhm} is not established to be compact. However, note that since it is relatively closed in $\partial \Lambda \setminus \partial \Omega$, it follows that it is not dense in $\partial \Lambda \setminus \partial \Omega$, and this excludes a potential pathological scenario. Indeed, for $z \in \partial \Lambda$ and $R>0$ such that $\overline{B_R(z)} \cap \partial \Omega = \emptyset$, (\ref{erd}) implies $$S \cap \overline{B_R(z)} \subset \big(G \cap \overline{B_R(z)}\big) \cup \big(\nabla \tilde \Psi_m(S_1)\cap \overline{B_R(z)}\big).$$ However, by combining Proposition \ref{geom2b}, Lemma \ref{yay}, Corollary \ref{sigc}, and Proposition \ref{propo23} one may deduce $$\mathcal{H}^{n-2}\Big(\big(G \cap \overline{B_R(z)}\big) \cup \big(\nabla \tilde \Psi_m(S_1)\cap \overline{B_R(z)}\big)\Big)<\infty.$$ To see this, note that since $\nabla \tilde \Psi_m(A_2)\cap \overline{B_R(z)}$ stays away from the common region, it is $\mathcal{H}^{n-2}$ finite by Proposition \ref{geom2b} and Lemma \ref{yay}. Next, by Corollary \ref{sigc} we know that $\nabla \tilde \Psi_m(A_1)$ is relatively closed in $\partial \Lambda \setminus \partial \Omega$; hence, $\nabla \tilde \Psi_m(A_1)\cap \overline{B_R(z)}$ is compact and so we may extract a finite subcover from (\ref{covr2}) to deduce $\mathcal{H}^{n-2}\big(\nabla \tilde \Psi_m(A_1)\cap \overline{B_R(z)}\big)<\infty,$ and as $$\nabla \tilde \Psi_m(S_1)=\nabla \tilde \Psi_m(A_1)\cup \nabla \tilde \Psi_m(A_2),$$ this shows the $\mathcal{H}^{n-2}$ finiteness of $\nabla \tilde \Psi_m(S_1)\cap \overline{B_R(z)}$. Last, note that $$G \cap \overline{B_R(z)}=\big(G_1 \cap \overline{B_R(z)}\big)\cup \big(G_2 \cap \overline{B_R(z)}\big);$$ $G_2$ is $\mathcal{H}^{n-2}$ finite by Proposition \ref{propo23} and $G_1$ is relatively closed in $\partial \Lambda \setminus \partial \Omega$. This implies that $G_1 \cap \overline{B_R(z)}$ is compact, and so extracting a finite subcover in (\ref{covr4}) proves $$\mathcal{H}^{n-2}\big(G_1 \cap \overline{B_R(z)}\big)<\infty.$$                    


\end{rem}

\begin{rem}
Note that to prove Theorem \ref{thhm}, we needed a $C^{1,1}$ regularity assumption on the domains $\Omega$ and $\Lambda$. This regularity was used in the proof of Proposition \ref{geom2b}.
\end{rem}

\section{Open problems}
\label{Section5}

\noindent 1. In Theorem \ref{thhm}, we proved that the free boundary is locally $C^{1,\alpha}$ away from $\partial(\Omega \cap \Lambda)$ and a singular set at the intersection of the free boundary with the fixed boundary. Hence, if one would be able to show that the free boundary stays away from $\partial(\Omega \cap \Lambda)$ inside the supports, then it would follow that the free boundary is locally $C^{1,\alpha}$ inside the supports; indeed, Figalli has already established that the free boundaries are globally H\"{o}lder continuous \cite[Theorem 1]{Fi2}; thus, one could improve his result by proving that such intersections do not happen and applying Corollary \ref{cor2}. Moreover, if one can show that the free boundary stays away from $\partial(\Omega \cap \Lambda)$ altogether, it would also follow that the singular set $S$ of Theorem \ref{thhm} is compact (see Remark \ref{imprem}); hence, $\mathcal{H}^{n-2}$ - finite (instead of $\sigma$-finite). A counterexample in which the free boundary touches the common region would also be enlightening, indicating that singularities may very well exist.
\vskip .2in 
\noindent 2. In Corollary \ref{cor1}, we proved that the partial transport is locally $C^{1,\alpha}$ away from some singular set $F$. By using the fact that free boundary never maps to free boundary (except possibly at the intersection of fixed with free boundary), we were able to estimate the Hausdorff dimension of a portion of $F$, which showed up in the form of $S$ in Theorem \ref{thhm}, and since the normal to the free boundary is in the direction of transport, we were able to deduce some regularity on the free boundary. However, the entire singular set of the partial transport is still not quite understood. Indeed, the set $\nabla \tilde \Psi_m(\partial_{nc} U_m) \subset S$ emerged in the course of proving that the Monge-Amp\`{e}re measure associated to the partial transport is a doubling measure, see Lemma \ref{b}. Perhaps one can improve this lemma by replacing the set $\partial_{nc} U_m$ with the set of points for which the Monge-Amp\`{e}re does not double affinely. Since the free boundaries are semiconvex \cite[Proposition 4.5]{Fi}, it may be possible to exploit the geometry to obtain estimates on its Hausdorff dimension; this gets into the regularity theory for the Monge-Amp\`{e}re equation (up to the boundary) on semiconvex domains.  
\vskip .2in          
\noindent 3. In the course of our study, we proved that certain subsets of the singular set were empty (see e.g. Proposition \ref{propo}). Therefore, it is natural to wonder whether the singular set $S$ in Theorem \ref{thhm} is empty.
\vskip .2in
\noindent 4. In Theorem \ref{thhm}, we assumed that $\Omega$ and $\Lambda$ were $C^{1,1}$ and uniformly convex domains. This was utilized in the proof of Proposition \ref{geom2b}, which is a purely geometric statement about two convex sets. Therefore, a natural line of research would be to reduce the $C^{1,1}$ regularity assumption in Proposition \ref{geom2b} (it seems plausible for the statement to be true under only a strict convexity assumption). As an application one could thereby utilize the method we developed in \S \ref{Section4} to improve Theorem \ref{thhm}.\\

\noindent \textbf{Acknowledgments.} The author wishes to thank Alessio Figalli for intellectually stimulating discussions and for his careful remarks on a preliminary version of the paper. The author is also grateful to two anonymous referees for their valuable comments and suggestions that significantly improved the presentation of the paper. This research was partially supported by an NSF RTG fellowship for graduate studies at the University of Texas at Austin and by an NSF EAPSI fellowship. Moreover, the excellent research environment provided by MSI at the Australian National University during various stages of this work is also kindly acknowledged.

\section{Appendix}

\begin{proof}[\bf{Proof of Theorem \ref{th1}}] 
Assume $x \in \overline{\Omega \cap U_m}$ and $R>0$ is such that $$B_{3R}(x) \cap \big(\overline{\Lambda \setminus U_m} \cup X_s \big) = \emptyset,$$  with $B_{3R}(x)\cap \overline{\Omega \cap U_m}$ convex. Let $z_0 \in \overline{\Omega \cap U_m} \cap B_R(x)$ so that $B_R(z_0) \subset B_{2R}(x)$.  Thus, $$B_{R}(z_0) \cap \big(\overline{\Lambda \setminus U_m} \cup X_s \big) = \emptyset.$$ Note also that $B_{R}(z_0)\cap \overline{\Omega \cap U_m}$ is convex. Since $\nabla \tilde \Psi_m(x)=x$ on $\Lambda \setminus \overline{V_m}$, we have $$\overline{\partial V_m \cap \Lambda} \cap \partial(\Omega \cap \Lambda) \subset \big(\partial \Omega \cap \partial \Lambda \cap \{\nabla \tilde \Psi_m(z)=z\} \cup (\partial V_m \cap \partial \Omega \cap \Lambda)\big)\subset X_s,$$ and since $B_r(z_0)$ is disjoint from $X_s$ it follows that $z_0 \in X$ ($X$ is defined in Lemma \ref{b}). By Lemma \ref{b}, $B_R(z_0)$ forms a doubling neighborhood around $z_0$. Now Lemma \ref{d} tells us that $\lim_{\epsilon \rightarrow 0} Z_{\epsilon}(z_0)=\{z_0\}$ in Hausdorff distance. So for $R>0$ pick $\epsilon_0>0$ as in Lemma \ref{d} so that 
\begin{equation} \label{eq1}
Z_{s\epsilon_0}(z_0) \subset B_R(z_0) \hskip .1in \forall s \in [0,1]. 
\end{equation}
Note that this analysis was valid for any $z_0 \in \overline{\Omega \cap U_m} \cap B_R(x)$ and the $\epsilon_0$ only depends on $R$. We use this in the following claim.   

\vskip .2 in 
\textit{Claim:} Let $t \in (0,1)$ so that $\frac{t}{1-t}=n^{\frac{3}{2}}$ and choose any $t_2 \in (t,1)$. Let $s_0=s_0(t_2,1)$ be the constant from Lemma \ref{c}. Then for all $\epsilon \in (0,\epsilon_0]$, $z_0 \in  \overline{\Omega \cap U_m} \cap B_R(x)$, and $z_1 \in  \overline{\Omega \cap U_m} \cap B_R(x) \cap \partial Z_{\epsilon}(z_0)$ we have 
\begin{equation} \label{eq-1}
\tilde \Psi_m(z_1)\geq \tilde \Psi_m(z_0)+\langle \nabla \tilde \Psi_m(z_0), z_1-z_0\rangle + \frac{\epsilon}{t} s_0(t_2,1).
\end{equation}   
\vskip .2in 
\textit{Proof of claim:} Without loss of generality, assume $\nabla \tilde \Psi_m(z_0)=0$ and let $z_t:=(1-t)z_0+tz_1.$ Since $z_1 \in  \overline{\Omega \cap U_m} \cap B_R(x) \cap \partial Z_{\epsilon}(z_0)$, it follows that $z_t \in t \cdot \overline{Z_\epsilon}(z_0)$. We would like to apply Lemma \ref{c}, so take $z \in spt M_{\tilde \Psi_m} \cap Z_{\epsilon}(z_0) = \overline{\Omega \cap U_m} \cap Z_{\epsilon}(z_0)$. By (\ref{eq1}), we know $Z_{\epsilon}(z_0) \subset B_R(z_0)$ (pick $s=\frac{\epsilon}{\epsilon_0} \leq 1$) so we have that $z \in B_R(z_0) \subset B_{2R}(x)$. Therefore, $B_{R}(z) \subset B_{3R}(x)$ and since $B_{3R}(x) \cap X_s = \emptyset$, it follows that $B_{R}(z) \cap X_s = \emptyset$. Thus, by Lemma \ref{d}, we have that $Z_{s\epsilon}(z) \subset B_R(z)$ for all $s\in [0,1]$. Note that $$B_{R}(z) \cap spt M_{\tilde \Psi_m} = B_R(z) \cap (B_{3R}(x) \cap \overline{\Omega \cap U_m});$$ hence, $B_{R}(z) \cap spt M_{\tilde \Psi_m}$ is convex and still disjoint from $\overline{\partial V_m \cap \Lambda} \cap \partial(\Omega \cap \Lambda)$ (since it is disjoint from $X_s$), therefore, by Lemma \ref{b}, $B_R(z)$ is a doubling neighborhood around $z$ and since $Z_{s\epsilon}(z) \subset B_R(z)$ is a convex body for all $s \in [0,1]$, we satisfy the doubling assumption of Lemma \ref{c} (note that the doubling constant is universal depending only on the initial data). Therefore, we have that if $$\tilde z \in spt M_{\tilde \Psi_m} \cap t_2 \cdot Z_{\epsilon}(z_0) = \overline{\Omega \cap U_m}\cap t_2 \cdot Z_{\epsilon}(z_0),$$ then $Z_{s_0 \epsilon}(\tilde z) \subset Z_{\epsilon}(z_0)$. Now $z_0, z_1\in B_R(x)\cap \overline{\Omega \cap U_m}$, and by assumption, $B_{3R}(x)\cap \overline{\Omega \cap U_m}$ is convex (hence, $B_R(x)\cap \overline{\Omega \cap U_m}$ is convex); thus, it follows that $z_t \in \overline{\Omega \cap U_m}$. Hence, since $z_t \in \overline{\Omega \cap U_m} \cap t \cdot \overline{Z_\epsilon}(z_0) \subset \overline{\Omega \cap U_m} \cap t_2 \cdot Z_\epsilon(z_0),$ we obtain $$Z_{s_0 \epsilon}(z_t) \subset Z_{\epsilon}(z_0).$$ In particular, since $z_1 \in \partial Z_{\epsilon}(z_0)$, $z_1$ is not an interior point of $Z_{s_0 \epsilon}(z_t)$. We also claim that $z_0$ is not an interior point. Indeed, suppose that this is not the case and let $x:=z_0-z_t$. Since $z_t+x=z_0 \in Z_{s_0 \epsilon}(z_t)$, by John's Lemma \cite[Lemma 2]{C2} we have $z_t-\frac{x}{\alpha} \in Z_{s_0 \epsilon}(z_t)$ with $\alpha:=n^{\frac{3}{2}}$. But $x=t(z_0-z_1)$ and $\alpha=\frac{t}{1-t}$ so that $z_t-\frac{x}{\alpha} = z_t-(1-t)(z_0-z_1)=z_1$, a contradiction to the fact that $z_1$ is not an interior point of $Z_{s_0 \epsilon}(z_t)$. Thus, neither $z_0$ nor $z_1$ are interior points of $Z_{s_0 \epsilon}(z_t)=\{ x: \tilde \Psi_m(x) < L_{\epsilon}(x):=\tilde \Psi_m (z_t)+\langle \nu_{\epsilon}, x-z_t\rangle + s_0 \epsilon \}$, so 
\begin{equation} \label{eq2}
\tilde \Psi_m(z_0) \geq L_{\epsilon}(z_0),
\end{equation} 

\begin{equation} \label{eq3}
\tilde \Psi_m(z_1) \geq L_{\epsilon}(z_1). 
\end{equation} 
Now since $\tilde \Psi_m$ is convex, $$\tilde \Psi_m(z_0)+\langle \nabla \tilde \Psi_m(z_0), w-z_0\rangle \leq \tilde \Psi_m(w) \hskip .1in \forall w \in \mathbb{R}^n.$$ But $\nabla \tilde \Psi_m(z_0)=0$ and letting $w=z_t$ we readily obtain $\tilde \Psi_m(z_0) \leq  \tilde \Psi_m(z_t)$. Therefore, by combining this information with $(\ref{eq2})$, $$L_\epsilon(z_t)=\tilde \Psi_m(z_t)+\epsilon s_0 \geq \tilde \Psi_m(z_0)+\epsilon s_0 \geq L_\epsilon(z_0)+\epsilon s_0.$$ This implies that on the line segment from $z_0$ to $z_t$, the slope of $L_\epsilon$ is at least $\frac{\epsilon s_0}{|z_t-z_0|}$. In particular, $$L_\epsilon(z_1) \geq L_\epsilon(z_t) + \bigg(\frac{\epsilon s_0}{|z_t-z_0|}\bigg)|z_1-z_t|.$$ Now, using (\ref{eq3}) and that $\frac{|z_1-z_t|}{|z_t-z_0|}=\frac{1-t}{t}$, we obtain 
\begin{align*}
\tilde \Psi_m(z_1)& \geq L_{\epsilon}(z_1) \geq L_\epsilon(z_t) + \bigg(\frac{\epsilon s_0}{|z_t-z_0|}\bigg)|z_1-z_t|\\
&\geq \tilde \Psi_m(z_0)+\epsilon s_0+(\epsilon s_0)\bigg(\frac{1-t}{t}\bigg)=\tilde \Psi_m(z_0)+\frac{\epsilon s_0}{t},
\end{align*}            
which proves the claim. 

\textit{End of claim.}\\
Now we are ready to prove the theorem. Let $r=r(R, \epsilon_0)$ be the constant from \cite[Lemma A.5]{CM} and let $z_0 \neq z_1 \in \Omega \cap U_m \cap B_{\frac{r}{2}}(x)$; using (\ref{eq1}), we may apply \cite[Lemma A.5]{CM} to obtain $z_1 \in B_r(z_0) \subset Z_{\epsilon_0}(z_0)$. By Lemma \ref{d} and \cite[Lemma A.8]{CM}, we know that $Z_\xi(z_0)$ is continuous in the variable $\xi$ and converges uniformly to $z_0$ as $\xi \rightarrow 0$. This implies the existence of $\epsilon \in (0,\epsilon_0)$ so that $z_1 \in \partial Z_{\epsilon}(z_0)$. Now choose any $\bar t \in (0,1)$ and let $s_0(0,\bar t) \in (0,1)$ be the corresponding constant from Lemma \ref{c}. Observe that by Corollary \ref{core1}, $s < s_0(0,\bar t)^k$ implies    
$$Z_{s\epsilon_0}(z_0) \subset \bar t^k \cdot Z_{\epsilon_0}(z_0)$$ for $k \in \mathbb{N}$. Let $s:=\frac{\epsilon}{\epsilon_0}$ and note that by the uniform convergence of the sections, 
up to possibly replacing $r$ with some $\tilde r<r$ depending on $t_0, \epsilon_0,$ and the initial data, we may assume $s < s_0(0,\bar t)$ so that there exists $k \in \mathbb{N}$ for which $\frac{\log(s)}{\log(s_0(0,\bar t))} \in [k,k+1)$; in particular, $s < s_0(0,\bar t)^k$ and since $z_1 \in \partial Z_{\epsilon}(z_0)=\partial Z_{s\epsilon_0}(z_0),$ it follows that $z_1 \in \bar t^k \cdot \overline{Z_{\epsilon_0}(z_0)}$.       
Hence, $z_1=(1-\bar t^k)z_0+\bar t^k w$ for some $w \in \overline Z_{\epsilon_0}(z_0)$. Moreover, $$|z_1-z_0|=\bar t^k|w-z_0| \leq \bar t^{(\frac{\log(s)}{\log(s_0(0,\bar t))}-1)} |w-z_0|=s^{\frac{\log(\bar t)}{\log(s_0(0,\bar t))}} \frac{|w-z_0|}{\bar t}.$$ But by (\ref{eq1}), $w \in \overline{B_R(z_0)}$ so that $|w-z_0| \leq R$. Hence,  using the definition of $s$, 
\begin{equation} \label{eq4}
|z_1-z_0| \leq \gamma(\bar t,\epsilon_0,R) \epsilon^{\frac{\log(\bar t)}{\log(s_0(0,\bar t))}},
\end{equation}
for some explicit constant  $\gamma(\bar t,\epsilon_0,R)$. Now the convexity of $\tilde \Psi_m$ yields $$\tilde \Psi_m(z_0) \geq \tilde \Psi_m(z_1)+ \langle \nabla \tilde \Psi_m(z_1), z_0-z_1\rangle,$$ so that by combining this inequality with (\ref{eq-1}) and using (\ref{eq4}), we obtain $$\langle \nabla \tilde \Psi_m(z_1)- \nabla \tilde \Psi_m(z_0), z_1-z_0\rangle \geq \frac{\epsilon}{t} s_0(t_2,1) \geq C|z_1-z_0|^{\frac{\log(s_0(0,\bar t))}{\log(\bar t)}},$$
where $C=C(\epsilon_0,R,t,\bar t, t_2)>0$. Note that since $r=\beta \epsilon_0^{\frac{n}{2}}R^{1-n}$, by picking $\epsilon_0$ smaller if necessary, we may assume without loss of generality that $|z_1-z_0| < 1$. Therefore, we may take $p:=\max \bigl\{\frac{\log(s_0(0,\bar t))}{\log(\bar t)},2 \bigr\}$ as the convexity exponent.                           
\end{proof}

\signei

\end{document}